%% file: main.tex
\documentclass{article}
\usepackage{microtype}
\usepackage{graphicx}
\usepackage{subfigure}
\usepackage{booktabs}
\usepackage{hyperref}

\usepackage[accepted]{icml2018}

\icmltitlerunning{On Acceleration with Noise-Corrupted Gradients}
\input{prologue.tex}
\begin{document}
\onecolumn
\icmltitle{On Acceleration with Noise-Corrupted Gradients}



\icmlsetsymbol{equal}{*}

\begin{icmlauthorlist}
\icmlauthor{Michael B. Cohen}{mit}
\icmlauthor{Jelena Diakonikolas}{bu}
\icmlauthor{Lorenzo Orecchia}{bu}
\end{icmlauthorlist}

\icmlaffiliation{mit}{Department of EECS, Massachusetts Institute of Technology, Cambridge, MA, USA}
\icmlaffiliation{bu}{Department of Computer Science, Boston University, Boston, MA, USA}

\icmlcorrespondingauthor{Jelena Diakonikolas}{jelenad@bu.edu}
\icmlcorrespondingauthor{Lorenzo Orecchia}{orecchia@bu.edu}

\icmlkeywords{First-order methods, acceleration, noise, gradients}

\vskip 0.3in



\printAffiliationsAndNotice{} 

\begin{abstract}
Accelerated algorithms have broad applications in large-scale optimization, due to their generality and fast convergence. 
However, their stability in the practical setting of noise-corrupted gradient oracles is not well-understood. 
This paper provides two main technical contributions: (i) a new accelerated method \agdp~that generalizes Nesterov's \agd~and improves on the recent method \axgd~\cite{AXGD}, and (ii) a theoretical study of accelerated algorithms under noisy and inexact gradient oracles, which is supported by numerical experiments.  This study leverages the simplicity of \agdp~and its analysis
to clarify the interaction between noise and acceleration and to suggest modifications to the algorithm that reduce the mean and variance of the error incurred due to the gradient noise.
\end{abstract}

\section{Introduction}\label{sec:intro}

First-order methods for convex optimization play a fundamental role in the solution of modern large-scale computational problems, encompassing applications in machine learning~\cite{Bube2014}, scientific computing~\cite{ST04,KOSZ13} and combinatorial optimization~\cite{Sherman2017, Ene}.
%
%
A central object of study in this area is the notion of {\it acceleration} -- an algorithmic technique that can be deployed when minimizing a {\it smooth} convex function $f(\cdot)$ via queries to a first-order oracle (a {\it blackbox}~that on input $\vx \in \mathcal{X}$, returns the vector $\nabla f(\vx)$ in constant time). In this setting, a function $f(\cdot)$ is $L$-smooth if it is differentiable and its gradient is $L$-Lipschitz continuous w.r.t to a pair of dual norms $\|\cdot \|,\, \|\cdot\|_*$, i.e.:
\begin{equation}\label{eq:smoothness}
\forall \; \vx,\vy \in \mathcal{X}, \; \|\nabla f(\vx) - \nabla f(\vy) \|_* \leq L \cdot \|\vx-\vy\|.
\end{equation}
Notably, the idea of acceleration can be generalized beyond this notion of smoothness to various weakly smooth problems. Examples include problems in which $f(\cdot)$ has H\"older-continuous gradients, and even the problems with  certain structured non-smooth objectives~\cite{nesterov2005smooth, Allen-Zhu2015,lu2016relatively}. In this paper, we restrict our attention to the original smooth setting, to which all others can be traced back.

Acceleration is interesting because it yields faster algorithms than classical steepest-descent algorithms, often matching or closely approximating known information-theoretic lower bounds on the number of necessary queries to the oracle. 
In the simplest smooth setting, the optimal accelerated algorithm, Accelerated Gradient Descent~\cite{nesterov1983},  achieves an error that scales as $O(1/k^2),$ where $k$ is the number of oracle queries. This should be compared to the convergence of steepest-descent methods, which attempt to locally minimize the first-order approximation to the function and only yield $O(1/k)$-convergence~\cite{ben2001lectures,nesterov2013introductory}.
Many of the workhorses of optimization, such as conjugate gradient and FISTA~\cite{beck2009fast}, are instantiations of accelerated algorithms.

Because of its generality, acceleration still proves an active topic of research. 
In particular, two weaknesses in the classical presentation of accelerated methods have recently attracted attention of scholars and practitioners alike: 1) the complexity and lack of underlying intuition in the convergence analysis of accelerated methods, and 2) the apparent lack of robustness to perturbations of the gradient oracle displayed by accelerated methods when compared to their non-accelerated counterparts.

Recently, some of the mystery of acceleration has faded, as different works have provided natural interpretations and alternative proofs for accelerated methods~\cite{AllenOrecchia2017,krichene2015accelerated,wibisono2016variational,Bubeck2015,lessard2016analysis,hu2017control,thegaptechnique}. Of particular interest to our work is the framework of~\cite{thegaptechnique}, which completely derives accelerated algorithms from the Euler discretization of a continuous dynamics that  minimizes a natural notion of duality gap. 

In terms of robustness, it has long been observed empirically that a na\"ive application of accelerated algorithms to inexact oracles often leads to error accumulation, even in the setting of random perturbations, while standard steepest descent algorithms do not suffer from this problem~\cite{Hardt-blog}. From a theoretical point of view, a number of papers have introduced oracle models that account for inexact gradient information. For example,~\cite{d2008smooth} proposed a restricted model of perturbations to the gradient that preserves the possibility of acceleration.  More recently,~\cite{devolder2014first} proposed a more general framework that allows for larger perturbations and seems to capture the error accumulation and instability observed in practice for accelerated methods. In  these works, the inexact oracle outputs an arbitrary deterministic perturbation of the true gradient oracle. 
In particular,~\cite{devolder2014first}~shows that such perturbations can be adversarially chosen to encode non-smooth problems. 

For stochastic perturbations, \cite{lan2012optimal,ghadimi2012optimal,ghadimi2013optimal} considered an additive-noise model, under which  \cite{lan2012optimal,ghadimi2012optimal} obtained an optimal convergence bound for the accelerated algorithm \acsa~in the smooth, non-strongly convex setting, but sub-optimal for the smooth, strongly-convex case.\footnote{In particular, the deterministic term in the convergence bound in~\cite{ghadimi2012optimal} decreases as $O(1/k^2)$ instead of the optimal $O(1-1/\sqrt{\kappa})^k$ convergence, where $\kappa$ is the objective function's condition number.} This bound was further improved to the optimal one in~\cite{ghadimi2013optimal} for the setting of \emph{constrained} smooth and strongly convex minimization, by coupling \acsa~algorithm from~\cite{ghadimi2012optimal} with a domain-shrinking procedure. More recently, \cite{Jain2017} completely closed this gap for the case of linear regression. Additionally,~\cite{dvurechensky2016stochastic} unified the deterministic model~\cite{devolder2014first}, the stochastic model~\cite{ghadimi2012optimal}, and the associated results. These references are the most closely related to our work.





\paragraph{Our contributions}
We study the issue of robustness of accelerated methods in three steps. First, we propose a novel, simple, generic accelerated algorithm \agdp~following the framework of~\cite{thegaptechnique}. 
This algorithm has a simple interpretation and analysis, and generalizes other known accelerated algorithms. 

Second, we leverage the simplicity of the analysis of \agdp~to characterize its behavior on different models of inexact oracles. Our analysis recovers the results for the deterministic oracle models of~\cite{d2008smooth} and~\cite{devolder2014first}. More generally, we consider the more general model of noise-corrupted gradient oracle, in which the true gradient $\nabla f(\vx)$ is corrupted by additive noise $\veta$:
\begin{equation}\label{eq:noisy}
\tnabla f(\vx) = \nabla f(\vx) + \veta,
 \end{equation}
where the perturbation $\veta$ may be a random variable. Such a model captures the setting of stochastic methods, in which the gradient is only estimated from a subset of its components~\cite{lan2012optimal,ghadimi2012optimal,ghadimi2013optimal,atchade2014stochastic,krichene-17acc-avg,Jain2017}, the setting of differentially private empirical risk minimization, in which Gaussian noise is intentionally added to the gradient to protect the privacy of the data~\cite{smith}, and the setting of engineering systems in which the gradient is estimated from noisy measurements~\cite{birand2013measurements}.

Our algorithm \agdp~is closely related to \acsa~from \cite{lan2012optimal} and can in fact be seen as a ``lazy'' (dual averaging) counterpart of \acsa.
After this paper had been submitted, Gasnikov and Nesterov independently proposed a universal method for stochastic composite optimization~\cite{Gasnikov2018}. While their algorithm is defined recursively and does not explicitly account for the iterative construction of a dual solution, a simple unwinding of the recursion shows that it is identical to \agdp.
However, the fact that \agdp~is obtained and analyzed through the use of the approximate duality gap technique~\cite{thegaptechnique} allows us to streamline the analysis and obtain various bounds for both deterministic and stochastic models of noise. Further, in the setting of smooth and strongly convex minimization, our analysis leads to a  tighter convergence bound for a single-stage algorithm (without domain-shrinking) than previously obtained in~\cite{ghadimi2012optimal,ghadimi2013optimal}~(see Section~\ref{app:ssc-magdp} for a precise statement).

{There are other models of noise that are not considered here. For example, we do not consider the model that includes both multiplicative and additive error in the gradient oracle~\cite{hu2017analysis}. Further, stochastic methods with variance reduction (see, e.g.,~\cite{schmidt2017minimizing,allen2017katyusha} and references therein) lead to a particular structure of the gradient noise variance (e.g., Lemma~3.4 in~\cite{allen2017katyusha}) that is not explored in this work. Nevertheless, we believe that our analysis  is general enough to be extended to these settings as well, which is deferred to the future version of this paper.}


Our results reveal an interesting discrepancy between noise tolerance in the settings of constrained and unconstrained smooth minimization. Namely, in the setting of constrained optimization, the error due to noise does not accumulate and is proportional to the diameter of the feasible region and the expected norm of the noise. In the setting of unconstrained optimization, the bound on the error incurred due to the noise accumulates, as observed empirically by~\cite{Hardt-blog}. However, our analysis also suggests a simple restart and slow down semi-heuristic for stabilizing the noise-incurred error, which allows taking advantage of both the acceleration and the noise stability under stochastic noise.

In the case of smooth and strongly convex minimization (Section~\ref{app:ssc-magdp}), the error due to noise does not accumulate even if the region is unconstrained, as long as the noise is zero-mean, independent, and has bounded variance.\footnote{Obtaining similar bounds for a slightly more general model that relaxes independence (similar to~\cite{lan2012optimal,ghadimi2012optimal}) is also possible; see Section~\ref{app:gen-stochastic-models}.} Further, using smaller step sizes than in the standard accelerated version of the method, the error due to noise decreases at rate $1/k$ (compare this to the $1/\sqrt{k}$ rate for smooth non-strongly convex functions). This means that strong convexity of a function implies higher robustness to noise. 


Finally, we verify the predictions and insights from our analysis of \agdp~by performing numerical experiments comparing \agdp~to other accelerated and non-accelerated methods on noise-corrupted gradient oracles. 
A noteworthy outcome of these experiments is the following:  when a natural generic  restart \& slow-down semi-heuristic is applied, the accelerated algorithm \axgd~\cite{thegaptechnique} and the algorithm \agdp~presented in this paper seem to outperform Nesterov's \agd~both in expectation and in variance in the presence of large noise. Further, we note that compared to \axgd, \agdp~reduces the oracle complexity (the number of queried gradients) by a factor of two.

%





\section{Notation and Preliminaries}\label{sec:prelims}

We assume that we are given a continuously differentiable convex function $f:\mathcal{X}\rightarrow \mathbb{R}$, where $\mathcal{X} \subseteq \mathbb{R}^n$ is a closed convex set. Hence:
\begin{equation}\label{eq:def-cvx}
\forall \vy, \vx \in \mathcal{X}: \; f(\vy) \geq f(\vx) + \innp{\nabla f(\vx),\, \vy - \vx},
\end{equation}
where $\nabla f(\cdot)$ denotes the gradient of $f(\cdot)$. 

Given oracle access to (possibly noise-corrupted) gradients of $f(\cdot)$, we are interested in minimizing $f(\cdot)$. We denote by $\vx_* \in \arg\min_{\vx \in \mathcal{X}} f(\vx)$ any (fixed) minimizer of $f(\cdot)$. 


We assume that there is an arbitrary (but fixed) norm \mbox{$\|\cdot\|$} associated with the space, and all the statements about function properties are stated with respect to that norm. We also define the dual norm $\|\cdot\|_*$ in the standard way: $\|\vz\|_* = \sup\{{\innp{\vz, \vx}}: \|\vx\|= 1\}$. The following definitions will be useful in our analysis, and thus we state them here for completeness.  

\begin{definition}\label{def:smoothness}
A function $f:\mathcal{X}\rightarrow \mathbb{R}$ is $L$-smooth on $\mathcal{X}$ with respect to a norm $\|\cdot\|$, if for all $\vx, \vxh \in \mathcal{X}$: $f(\vxh) \leq f(\vx) + \innp{\nabla f(\vx), \vxh - \vx} + \frac{L}{2}\|\vxh - \vx\|^2$.  This is equivalent to Equation~\eqref{eq:smoothness}.
\end{definition}
A gradient step is defined in a standard way as $\mathrm{Grad}(\vx) = \arg\min_{\vxh \in \mathcal{X}} \{f(\vx) + \innp{\nabla f(\vx), \vxh - \vx} + \frac{L}{2}\|\vxh-\vx\|^2\}$.
\begin{definition}\label{def:strong-convexity}
A function $f:\mathcal{X}\rightarrow \mathbb{R}$ is $\mu$-strongly convex on $\mathcal{X}$ with respect to a norm $\|\cdot\|$, if for all $\vx, \vxh \in \mathcal{X}$: $f(\vxh) \geq f(\vx) + \innp{\nabla f(\vx), \vxh - \vx} + \frac{\mu}{2}\|\vxh - \vx\|^2$.
\end{definition}

\begin{definition}\label{def:cxv-conj}(Convex Conjugate) Function $\psi^*$ is the convex conjugate of $\psi: \mathcal{X} \rightarrow \mathbb{R}$, if $\psi^* (\vz) = \max_{\vx \in \mathcal{X}}\{\innp{\vz, \vx} - \psi(\vx)\}$, $\forall \vz \in \mathbb{R}$. 
\end{definition}
We assume that there is a strongly-convex differentiable function $\psi:\mathcal{X}\rightarrow \mathbb{R}$ such that $\max_{\vx \in \mathcal{X}} \{\innp{\vz, \vx} - \psi (\vx)\}$ is easily solvable, possibly in a closed form. Notice that this problem defines the convex conjugate of $\psi(\cdot)$, i.e., $\psi^* (\vz) = \max_{\vx \in \mathcal{X}} \{\innp{\vz, \vx} - \psi (\vx)\}$. 
The following fact is a simple corollary of Danskin's Theorem\footnote{See, for instance, Proposition 4.15 in the textbook~\cite{Bertsekas2003}.}.
\begin{fact}\label{fact:danskin}
Let $\psi: \mathcal{X} \to \mathbb{R}$ be a differentiable strongly-convex function. Then:
$$
\nabla \psi^*(\vz) = \arg\max_{\vx \in \mathcal{X}} \left\{ \innp{\vz, \vx} - \psi(\vx)\right\}.
$$ 
\end{fact}

\begin{fact}\label{fact:smoothness-of-cvx-conj}
If $\psi(\cdot)$ is $\mu$-strongly convex w.r.t. a norm $\|\cdot\|$ for $\mu > 0$, then $\psi^*(\cdot)$ is $\frac{1}{\mu}$-smooth w.r.t. the norm $\|\cdot\|_*$.
\end{fact}

\begin{definition}\label{def:bregman-divergence}(Bregman Divergence)
$D_{\psi}(\vx, \vxh) \defeq \psi(\vx) - \psi(\vxh)-\innp{\nabla \psi(\vxh), \vx - \vxh}$, for $\vx \in \mathcal{X}, \vxh \in \mathcal{X}^o$, where $\mathcal{X}^o$ denotes the set of all points from $\mathcal{X}$ for which $\psi(\cdot)$ admits a (sub)gradient. 
\end{definition}
%
The Bregman divergence $D_{\psi}(\vx,\vy)$ captures the difference between $\psi(\vx)$ and its first order approximation at $\vy.$ Notice that, for a differentiable $\psi$, we have:
$
\nabla_{\vx} D_{\psi}(\vx,\vy) = \nabla \psi(\vx) - \nabla \psi(\vy).
$
The Bregman divergence $D_{\psi}(\vx,\vy)$ as a function $g_{\vy}(\vx)$ is convex. Its Bregman divergence is itself, i.e., $D_{g_{\vy}}(\vv, \vu) = D_{\psi}(\vv, \vu).$
\section{Improved Accelerated Method}\label{sec:algorithm}

In this section, we focus on the setting of smooth minimization. The case of smooth and strongly convex minimization is treated in Section~\ref{app:ssc-magdp}.

To design \agdp, we define an approximate duality gap, similar to~\cite{thegaptechnique,AXGD}, but allowing for  
an inexact gradient oracle  
according to Eq.~\eqref{eq:noisy}. 
The construction is based on maintaining three points at each iteration $k$: $\vx_k$ is the point at which the gradient is queried, while $(\vy_k,\vz_k)$ is the current primal-dual solution pair at the end of iteration $k.$ For this setup, the dual solution $\vz_k$ is a conic combination of the negative gradients seen so far, taken at an initial dual point $z_0 = \nabla \psi(\vx_0), $where $\vx_0$ is an arbitrary initial primal solution, i.e.,
\begin{align}\label{eq:z}
\vz_k = -\sum_{i=1}^k a_i \tnabla f(\vx_i) + \vz_0.
\end{align}
where the sequence $a_k > 0$, $A_k = \littlesum_{i=1}^k a_i$ will be specified later. By convention, $A_0 = 0$.

\subsection{Approximate Duality Gap}
The choice of sequences above immediately implies upper and lower bounds on optimum at each iteration $k.$ The upper bound is simply chosen as $U_k = f(\vy_k)$. For the lower bound, by convexity of $f(\cdot)$ (see Eq.~(\ref{eq:def-cvx})):
\begin{equation*}
f(\vx_*) \geq \frac{\littlesum_{i=1}^k a_i f(\vx_i) + \littlesum_{i=1}^k a_i \innp{\nabla f(\vx_i), \vx_* - \vx_i}}{A_k}.
\end{equation*}
To relate the lower bound to the output of the inexact oracle, it is useful to express the gradients $\nabla f(\vx_i)$ as $\nabla f(\vx_i) = \tnabla f(\vx_i) - \veta_i$. Adding and subtracting $\frac{1}{A_k} D_{\psi}(\vx_*, \vx_0)$ in the last equation, we have:
\begin{align*}
f(\vx_*) \geq  \frac{\littlesum_{i=1}^k a_i f(\vx_i) + \littlesum_{i=1}^k a_i \innp{\tnabla f(\vx_i), \vx_* - \vx_i}-\littlesum_{i=1}^k a_i\innp{\veta_i, \vx_* - \vx_i} + D_{\psi}(\vx_*, \vx_0) - D_{\psi}(\vx_*, \vx_0)}{A_k}.
\end{align*}
Finally, we can replace $\vx_*$ by a minimization over $\mathcal{X}$ to obtain our final lower bound:
\begin{align*}
f(\vx_*) &\geq \frac{\littlesum_{i=1}^k a_i f(\vx_i) - \littlesum_{i=1}^k a_i\innp{\veta_i, \vx_* - \vx_i} - D_{\psi}(\vx_*, \vx_0) + \min_{\vu \in \mathcal{X}}\left\{ \littlesum_{i=1}^k a_i \innp{\tnabla f(\vx_i), \vu - \vx_i} + D_{\psi}(\vu, \vx_0) \right\}}{A_k}\\
&\defeq L_k.
\end{align*}
Applying Fact~\ref{fact:danskin} and the definition of $\vz_k$ from (\ref{eq:z}), we have the following characterization of the last term of $L_k$.
\begin{proposition}\label{prop:arg-min-lb}
Let $\vz_k$ be defined as in~(\ref{eq:agdp}). Then:
\begin{align*}
\nabla\psi^*(\vz_k) = 
\arg \min_{\vu \in \mathcal{X}} &\Big\{ \littlesum_{i=1}^k a_i \innp{\tnabla f(\vx_i), \vu - \vx_i} + D_{\psi}(\vu, \vx_0) \Big\}.
\end{align*}
\end{proposition}

The approximate duality gap is simply defined as $G_k = U_k - L_k$. Observe that, by construction of $U_k$ and $L_k$, $f(\vy_k)-f(\vx_*)\leq G_k$. Hence, to prove the convergence of the algorithm, it suffices to bound $G_k$. To do so, we will track the evolution of the quantity $A_k G_k,$ i.e., we will bound\footnote{From~\cite{thegaptechnique} it can be derived that $A_k G_k$ is a Lyapunov function for the continuous dynamic underlying \agdp, i.e., $E_k$  is the discretization error at iteration $k$.}   $E_k = A_k G_k - A_{k-1}G_{k-1}$, so that 
$$
G_k = \frac{A_1}{A_k}G_1 + \frac{\littlesum_{i=2}^k E_i}{A_k}.
$$

\subsection{The \agdp Algorithm}
The steps of \agdp~are defined as follows:
\begin{equation}\label{eq:agdp}\tag{\agdp}
\begin{gathered}
\vx_k = \frac{A_{k-1}}{A_k} \vy_{k-1} + \frac{a_k}{A_k}\nabla\psi^*(\vz_{k-1}),\\
\vz_k = \vz_{k-1} - a_k \tnabla f(\vx_k),\\
\vy_k = \frac{A_{k-1}}{A_k} \vy_{k-1} + \frac{a_k}{A_k}\nabla\psi^*(\vz_k),
\end{gathered}
\end{equation}
To seed \agdp, we let $\vx_1 = \vx_0$, $\vy_1 = \vv_1 = \nabla\psi^*(\vz_1)$.

\paragraph{Related Algorithms} Compared to Nesterov's \agd, \agdp~differs in the sequence $\vy_k$: \agd~sets $\vy_k = \mathrm{Grad}(\vx_k)$. The two algorithms are equivalent when $\psi(\vx) = \frac{\mu}{2}\|\vx - \vx_0\|^2$, $\frac{{a_k}^2}{A_k} = \frac{\mu}{L}$, and $\mathcal{X} = \mathbb{R}^n$, but in general they produce different sequences of points. Thus, \agdp~can be seen as a generalization of \agd. Compared to a more recent accelerated method \axgd~of \cite{AXGD}, \agdp~differs in sequences $\vy_k$ and $\vz_k$. In particular, in \axgd, $\vz_k = \vz_{k-1} - a_k \tnabla f(\vy_k)$, while $\vy_k = \frac{A_{k-1}}{A_k} \vy_{k-1} + \frac{a_k}{A_k}\nabla\psi^*(\vz_k - a_k \tnabla f(\vx_k))$. As \axgd~uses the gradients of $f(\cdot)$ at both sequences $\vx_k$ and $\vy_k$ to define $\vx_k$ and $\vy_k$, it is more wasteful: its oracle complexity is twice as high as that of \agd~and \agdp. Most closely related to \agdp~is the \acsa~algorithm~\cite{lan2012optimal}; namely, for some step sizes, \agdp~can be seen as a ``lazy'' (dual averaging) version of \acsa. {The relationship between \magdp~(see Section~\ref{app:ssc-magdp}) and \acsa~for smooth and strongly convex minimization~\cite{ghadimi2012optimal} is not immediately clear, due to the different parameter choices.}

\subsection{Convergence Analysis for \agdp}

To simplify the notation, from now on we denote:
$$
\vv_k \defeq \nabla\psi^*(\vz_k).
$$
We can now bound the change $E_k = A_k G_k - A_{k-1}G_{k-1}$ by decomposing it into two terms: $E_k \leq E_k^e + E_k^{\eta}$, where the latter term is due to the inexact nature of the gradient oracle. The following lemma allows us to bound these terms. 

\begin{lemma}\label{lemma:total-error}
Let $E_k^{\eta} = \innp{\veta_k, \vx_* - \vv_k}$ and 
$E_k^e = A_k(f(\vy_k)-f(\vx_k))
- A_k\innp{\nabla f(\vx_k), \vy_k - \vx_k} - D_{\psi}(\vv_{k}, \vv_{k-1}).$ Then $E_k \leq E_k^{\eta} + E_k^e$.
\end{lemma}
\begin{proof}
Let $m_k(\vx) = \littlesum_{i=1}^k a_i \innp{\tnabla f(\vx_i), \vu - \vx_i} + D_{\psi}(\vu, \vx_0)$ denote the function under the minimum in the lower bound. By Proposition~\ref{prop:arg-min-lb}, $\vv_k = \nabla \psi^*(\vz_k) = \arg\min_{\vx \in \mathcal{X}}m_k(\vx)$. Observe that $m_k(\vx) = a_k \innp{\tnabla f(\vx_k), \vx - \vx_k} + m_{k-1}(\vx)$. By the definition of Bregman divergence:
\begin{align*}
 m_{k-1}(\vv_k) = m_{k-1}(\vv_{k-1})+ \innp{ \nabla m_{k-1}(\vv_{k-1}), \vv_k - \vv_{k-1}}
+ D_{m_{k-1}}(\vv_k, \vv_{k-1}).
\end{align*}
As Bregman divergence is blind to linear and zero-order terms, we have that  $D_{m_{k-1}}(\vv_k, \vv_{k-1}) = D_{\psi}(\vv_k, \vv_{k-1})$. By Proposition~\ref{prop:arg-min-lb}, $\vv_{k-1} = \arg\min_{\vx \in \mathcal{X}}m_{k-1}(\vx)$, and hence $\innp{ \nabla m_{k-1}(\vv_{k-1}), \vv_k - \vv_{k-1}}\geq 0$. Therefore, 
\begin{align*}
 m_k(\vv_k) \geq & m_{k-1}(\vv_{k-1})+ a_k \innp{\tnabla f(\vx_k), \vv_k - \vx_k}+ D_{\psi}(\vv_{k}, \vv_{k-1}).
\end{align*}
Using the definition of $\tnabla f(\vx_k)$, the change in the lower bound is:
\begin{equation}\label{eq:lb-change}
\begin{aligned}
A_k L_k - A_{k-1}L_{k-1} \geq a_k f(\vx_{k})+ a_k \innp{\nabla f(\vx_k), \vv_k - \vx_k}+ D_{\psi}(\vv_{k}, \vv_{k-1}) - a_k\innp{\veta_k, \vx_* - \vv_k}.
\end{aligned}
\end{equation}
For the change in the upper bound, we have:
\begin{equation}\label{eq:ub-change}
\begin{aligned}
A_k U_k - A_{k-1}U_{k-1} =& A_k f(\vy_k) - A_{k-1}f(\vy_{k-1})\\
=& a_k f(\vx_k) + A_k(f(\vy_k) - f(\vx_k))+ A_{k-1}(f(\vx_k)-f(\vy_{k-1})).
\end{aligned}
\end{equation}
By convexity of $f(\cdot)$:
\begin{equation}\label{eq:adgt-pf-convexity}
\begin{aligned}
f(\vx_k)-f(\vy_{k-1}) \leq & \innp{\nabla f(\vx_{k}), \vx_k - \vy_{k-1}}.
\end{aligned}
\end{equation}
Combining (\ref{eq:lb-change})-(\ref{eq:adgt-pf-convexity}) and (\ref{eq:agdp}):
\begin{align*}
A_kG_k - A_{k-1}G_{k-1} \leq  a_k \innp{\veta_k, \vx_* - \vv_k} - D_{\psi}(\vv_{k}, \vv_{k-1})+A_k(f(\vy_k)-f(\vx_k))
- A_k\innp{\nabla f(\vx_k), \vy_k - \vx_k}, 
\end{align*}
as claimed.
\end{proof}
%
The last piece that is needed for the analysis is the bound on the initial gap $G_1$, obtained in the following proposition.
\begin{proposition}
$A_1G_1 \leq D_{\psi}(\vx_*, \vx_0) + E_1^{\eta} + E_1^e$, where $E_1^{\eta}$ is defined as in Lemma~\ref{lemma:total-error} and $E_1^e = A_1(f(\vy_1) - f(\vx_1) - \innp{\nabla f(\vx_1), \vv_1 - \vx_1}) - D_{\psi}(\vv_1, \vx_0)$.
\end{proposition}
The proof is a straightforward application of the previously introduced definitions.

\subsection{Convergence of \agdp~with Exact Oracle }

To prove the convergence of the method in the noiseless case, in this section we assume that $\veta_k = \zeros$, and, consequently, $E_k^{\eta} = 0$. Hence, to obtain a convergence bound for \agdp, we only need to bound $E_k^e$. 
\begin{theorem}\label{thm:noiseless-convergence}
Let $f:\mathcal{X}\rightarrow \mathbb{R}$ be an $L$-smooth function and let $\vx_0 \in \mathcal{X}$ be an arbitrary initial point. If sequences $\vx_k, \vy_k, \vz_k$ evolve according to (\ref{eq:agdp}) for some $\mu$-strongly convex function $\psi(\cdot)$, $\veta_k = \zeros$, and $\frac{{a_k}^2}{A_k} \leq \frac{\mu}{L}$, then $\forall k \geq 1$:
\begin{equation*}
f(\vy_k) - f(\vx_*) \leq \frac{D_{\psi}(\vx_*, \vx_0)}{A_k}.
\end{equation*}
\end{theorem}
\begin{proof}
By smoothness of $f(\cdot)$, $f(\vy_k)-f(\vx_k) - \innp{\nabla f(\vx_k), \vy_k - \vx_k} \leq \frac{L}{2}\|\vy_k - \vx_k\|^2$. Hence:
\begin{equation*}
E_k = E_k^e \leq A_k \frac{L}{2}\|\vy_k - \vx_k\|^2 - D_{\psi}(\vv_k, \vv_{k-1}). 
\end{equation*}
From (\ref{eq:agdp}), $\vy_k - \vx_k = \frac{a_k}{A_k}(\vv_k - \vv_{k-1})$. As $D_{\psi}(\vv_k, \vv_{k-1})\geq \frac{\mu}{2}\|\vv_k - \vv_{k-1}\|^2$, it follows that:
\begin{equation*}
E_k \leq \frac{1}{2}\Big(\frac{{a_k}^2 L}{A_k} - {\mu}\Big)\|\vv_k - \vv_{k-1}\|^2 \leq 0,
\end{equation*}
as $\frac{{a_k}^2}{A_k} \leq \frac{\mu}{L}$ by the theorem assumptions. Thus: $G_k \leq \frac{A_1}{A_k}G_1$ and it remains to bound $A_1 G_1$, which is just:
\begin{align*}
A_1 G_1 = A_1 (f(\vy_1) - f(\vx_1) - \innp{\nabla f(\vx_1), \vv_1 - \vx_1})- D_{\psi}(\vv_1, \vx_0) + D_{\psi}(\vx_*, \vx_0)\leq D_{\psi}(\vx_*, \vx_0),
\end{align*}
as $\vy_1 = \vv_1$ and $\vx_1 = \vx_0$. 
\end{proof}
Observe that for $a_k = \frac{\mu}{L}\cdot \frac{k+1}{2}$ we recover the standard $1/k^2$ convergence rate of accelerated methods. 



\subsection{Convergence of \agdp~with Inexact Oracle}

In this subsection, we focus on bounding the error $E_k^{\eta}$ that is accrued due to the additive noise $\veta_k$. Additional results (including other models of noise) can be found in Section~\ref{sec:models}. From Lemma~\ref{lemma:total-error}, $E_k^{\eta} = a_k \innp{\veta_k, \vx_* - \vv_k}$, and we have the following:
\begin{proposition}\label{prop:noise-bnd-diam}
Let $f:\mathcal{X}\rightarrow \mathbb{R}$ be an $L$-smooth function and let $\vx_0 \in \mathcal{X}$ be an arbitrary initial point. If sequences $\vx_k, \vy_k, \vz_k$ evolve according to (\ref{eq:agdp}) for some $\mu$-strongly convex function $\psi(\cdot)$, where $\veta_k$'s are independent, $R_{\vx_*} = \max_{\vx\in \mathcal{X}}\|\vx - \vx_*\|$, and $\frac{{a_k}^2}{A_k} \leq \frac{\mu}{L}$, then $\forall k \geq 1$:
\begin{gather*}
\mathbb{E}\left[f(\vy_k) - f(\vx_*)\right] \leq \frac{D_{\psi}(\vx_*, \vx_0)}{A_k} + R_{\vx_*}\frac{\sum_{i=1}^k a_i\mathbb{E}[\|\veta_i\|_*]}{A_k},\\
\mathrm{Var}\left[ f(\vy_k) - f(\vx_*) \right] 
\leq {{R_{\vx_*}}^2}\frac{\sum_{i=1}^k {a_i}^2\mathbb{E}[\|\veta_i\|_*^2]}{{A_k}^2}.
\end{gather*}
\end{proposition}
\begin{proof}
From Theorem~\ref{thm:noiseless-convergence} and Lemma~\ref{lemma:total-error}:
\begin{equation}\label{eq:accumulate}
A_k G_k \leq D_{\psi}(\vx_*, \vx_0) + \sum_{i=1}^k a_i \innp{\veta_i, \vx_* - \vv_i}
\end{equation}
The bound on the expectation follows by applying $\innp{\vx, \vy} \leq \|\vx\|\|\vy\|_*$ (by the duality of norms), linearity of expectation, and $A_k = \littlesum_{i=1}^k a_i$. The bound on the variance follows by, in addition, using the standard facts that  $\mathrm{Var}[aX + bY] = a^2 \mathrm{Var}[X] + b^2 \mathrm{Var}[Y]$ and $\mathrm{Var}[X]\leq \mathbb{E}[X^2]$, where $a$, $b$ are constants and $X$, $Y$ are independent random variables.
\end{proof}
\begin{remark}
Observe that if $\mathbb{E}[\|\veta_i\|_*]\leq M$, $\mathbb{E}[\|\veta_i\|_*^2]\leq \sigma^2$, $\forall i$, then $\mathbb{E}\left[f(\vy_k) - f(\vx_*)\right] \leq \frac{D_{\psi}(\vx_*, \vx_0)}{A_k} + R_{\vx_*}M$ and $\mathrm{Var}\left[ f(\vy_k) - f(\vx_*) \right] 
\leq (R_{\vx_*}\sigma)^2\frac{\sum_{i=1}^k {a_i}^2}{{A_k}^2}$. The same bound on $\mathbb{E}\left[f(\vy_k) - f(\vx_*)\right]$ as in Prop.~\ref{prop:noise-bnd-diam} (and the special case stated here) holds \emph{even if $\veta_k$'s are not independent}.
\end{remark}

The bound from Proposition~\ref{prop:noise-bnd-diam} is mainly useful when $R_{\vx_*}$ is bounded, which is the case when, e.g., the diameter of $\mathcal{X}$ is bounded. For the case of unconstrained optimization (i.e., when $\mathcal{X} = \mathbb{R}^n$), the bound from Proposition~\ref{prop:noise-bnd-diam} is uninformative. Hence, we derive another bound that is independent of $R_{\vx_*}$, but it requires that the noise samples $\veta_k$ are both zero-mean and independent.\footnote{The assumption that $\veta_k$'s are independent can be relaxed -- see Remark~\ref{remark:indep-relax} below and Section~\ref{app:gen-stochastic-models}.}

\begin{lemma}\label{lemma:unconstr-noise-bnd}
Let $f:\mathcal{X}\rightarrow \mathbb{R}$ be an $L$-smooth function and let $\vx_0 \in \mathcal{X}$ be an arbitrary initial point. If sequences $\vx_k, \vy_k, \vz_k$ evolve according to (\ref{eq:agdp}) for some $\mu$-strongly convex function $\psi(\cdot)$, where $\veta_k$'s are zero-mean independent random variables and $\frac{{a_k}^2}{A_k} \leq \frac{\mu}{L}$, then $\forall k \geq 1$:
\begin{equation*}
\mathbb{E}\left[f(\vy_k) - f(\vx_*)\right] \leq \frac{D_{\psi}(\vx_*, \vx_0)}{A_k} + \frac{\sum_{i=1}^k {a_i}^2\mathbb{E}[\|\veta_i\|_*^2]}{\mu A_k}.
\end{equation*}
\end{lemma}
\begin{proof}
Let $\vvh_k = \nabla\psi^*(\vz_k + a_k \veta_k) = \nabla\psi^*(\vz_{k-1} - a_k \nabla f(\vx_k))$. Recall that $E_k^{\eta} = a_k\innp{\veta_k, \vx_* - \vv_k}$. Adding and subtracting $\vvh_{k}$:
\begin{equation*}
E_k^{\eta} = a_k\innp{\veta_k, \vx_* - \vvh_{k}} + a_k\innp{\veta_k, \vvh_{k}-\vv_k}.
\end{equation*}
As $\vvh_{k}$ is independent of $\veta_k$ and $\mathbb{E}[\veta_k] = \zeros$, $\mathbb{E}\left[\innp{\veta_k, \vx_* - \vvh_{k}}\right] = 0$. On the other hand, as $\vv_k = \nabla \psi^*(\vz_k)$ and $\psi^*(\cdot)$ is $\frac{1}{\mu}$-smooth by Fact~\ref{fact:smoothness-of-cvx-conj}, we have: 
$
\innp{\veta_k, \vvh_{k}-\vv_k} \leq \frac{a_k}{\mu}\|\veta_k\|_*^2.$ 
The rest of the proof is by Theorem~\ref{thm:noiseless-convergence} and Lemma~\ref{lemma:total-error}. 
\end{proof}

\begin{remark}\label{remark:indep-relax}
It is possible to relax the assumption that $\veta_k$'s are independent. In fact, for Lemma~\ref{lemma:unconstr-noise-bnd} to apply, it suffices that, conditioned on the natural filtration $\mathcal{F}_{k-1}$ (all the information about the noise up to the beginning of iteration $k$), $\veta_k$ is independent of $\vvh_k$. More details are provided in Section \ref{app:gen-stochastic-models}. 
\end{remark}

Lemma~\ref{lemma:unconstr-noise-bnd} suggests that for unconstrained smooth minimization the sequence $a_k$ that leads to accelerated methods aggregates noise, as for accelerated methods $a_k \sim k, A_k \sim k^2$. However, if we were to resort to a slower, uniform sequence (and slower $1/k$ convergence rate), then the noise would average out, as we would have constant $a_k$'s and $A_k \sim k$. Even more, if $a_k\sim 1/\sqrt{k}$, then the error due to noise would decrease at rate $\log(k)/\sqrt{k}$. This is confirmed by our numerical experiments and matches the experience of practitioners, as discussed by \cite{Hardt-blog}.

The lemma also shows that error accumulation can be avoided if we postulate that the magnitude of the noise vanishes with the number of iterations. This can be achieved if the estimates of the gradient improve over iterations {\cite{atchade2014stochastic}}. For example, if we have $\mathbb{E}[\|\veta_i\|_*^2] = O\Big(\frac{1}{a_i}\Big) = O\Big(\frac{1}{i}\Big)$, the noise-error term averages out, making the total error due to noise bounded. If $\mathbb{E}[\|\veta_i\|_*^2] = O\Big(\frac{1}{{i}^2}\Big)$, the noise-error term vanishes at rate $k/A_k = 1/k$. Finally, if $\mathbb{E}[\|\veta_i\|_*^2] = O\Big(\frac{1}{{i}^3}\Big)$, the noise-error term vanishes at rate $\log(k)/k^2$, essentially recovering accelerated convergence. 

Observe that we could not get a bound on variance that is independent of $R_{\vx_*}$, as the variance (unlike the expectation) of $\innp{\veta_k, \vx_* - \vvh_{k}}$ is not zero. Instead, since we upper-bound the expectation of $f(\vy_k)-f(\vx_*)$ by a non-negative quantity (and $f(\vy_k)-f(\vx_*)$ is always non-negative as $\vx_*$ is the minimizer of $f(\cdot)$), we can apply Markov's Inequality to obtain a concentration bound on $f(\vy_k)-f(\vx_*)$. 

Finally, the step sizes $a_k$ can be chosen so as to balance the deterministic error and the error due to noise in the convergence bound. This leads to the following corollary. 

\begin{corollary}
\label{cor:convergent-bound}
Let $f:\mathcal{X}\rightarrow \mathbb{R}$ be an $L$-smooth function and let $\vx_0 \in \mathcal{X}$ be an arbitrary initial point. If sequences $\vx_k, \vy_k, \vz_k$ evolve according to (\ref{eq:agdp}) for some $\mu$-strongly convex function $\psi(\cdot)$, where $\veta_k$'s are zero-mean independent random variables, $K \geq 1$ is an arbitrary (but fixed) number of iterations of \ref{eq:agdp}, $\gamma = \mu/\max\{L, \sqrt{\sum_{i=1}^K {b_i}^2\mathbb{E}[\|\veta_i\|_*^2]}\}$, $a_i = \gamma b_i$, and $\frac{{a_i}^2}{A_i} \leq \gamma$, $\forall i$, then:
\begin{equation*}
\mathbb{E}\left[f(\vy_K) - f(\vx_*)\right] \leq \frac{D_{\psi}(\vx_*, \vx_0)}{A_K} + \frac{\sqrt{\sum_{i=1}^K {a_i}^2\mathbb{E}[\|\veta_i\|_*^2]}}{A_K}.
\end{equation*}
In particular, if $b_i = \frac{i+1}{2}$ and, in addition, $\mathbb{E}[\|\veta_i\|_*^2]\leq \sigma^2$, $\forall i$, then:
\begin{align*}
\mathbb{E}\left[f(\vy_K) - f(\vx_*)\right]\leq \frac{4L D_{\psi}(\vx_*, \vx_0)}{\mu K(K+3)} + O\left(\frac{\sigma(\mu + {D_{\psi}(\vx_*, \vx_0)})}{\mu\sqrt{K}}\right).
\end{align*}
If $\mathbb{E}[\|\veta_i\|_*^2]\leq \sigma^2$, $\forall i$, but the value of $\sigma$ is unknown, then setting $b_i = \frac{i+1}{2}$ and $\gamma = \mu/\max\{L, \sqrt{\sum_{i=1}^K {b_i}^2}\}$ gives:
\begin{align*}
\mathbb{E}\left[f(\vy_K) - f(\vx_*)\right]\leq  O\left(\frac{\sigma^2}{\sqrt{K}}\right)+\max\left\{\frac{4L D_{\psi}(\vx_*, \vx_0)}{\mu K(K+3)}, O\Big(\frac{D_{\psi}(\vx_*, \vx_0)}{\mu \sqrt{K}}\Big)\right\}.
\end{align*}
\end{corollary}
\begin{proof}
By the choice of parameters, 
\begin{align*}
\mu = \gamma \max\Big\{L, \sqrt{\littlesum_{i=1}^K {b_i}^2\mathbb{E}[\|\veta_i\|_*^2]}\Big\}\geq \gamma\sqrt{\littlesum_{i=1}^K {b_i}^2\mathbb{E}[\|\veta_i\|_*^2]}= \sqrt{\littlesum_{i=1}^K {a_i}^2\mathbb{E}[\|\veta_i\|_*^2]},
\end{align*}
which bounds the stochastic term. 
The rest of the proof follows by plugging in particular choices of parameters and using that $\gamma L\leq \mu \leq \max\left\{\gamma L, \sqrt{\littlesum_{i=1}^K {a_i}^2\mathbb{E}[\|\veta_i\|_*^2]}\right\}\leq \gamma L + \sqrt{\littlesum_{i=1}^K {a_i}^2\mathbb{E}[\|\veta_i\|_*^2]}$.
\end{proof}
\begin{remark}\label{remark:convergent-bound}
Observe that the optimal choice of $\gamma$ to balance the terms from Lemma~\ref{lemma:unconstr-noise-bnd} would be $\gamma = \mu/ \max\left\{L, \sqrt{\frac{\sum_{i=1}^K {a_i}^2\mathbb{E}[\|\veta_i\|_*^2]}{D_{\phi}(\vx_*, \vx_0)}}\right\}$. When $\mathcal{X}$ is unconstrained, it is not always possible to estimate (an upper bound on) $D_{\phi}(\vx_*, \vx_0)$, which is why we made the particular choice of $\mu$ in Corollary~\ref{cor:convergent-bound}. However, when the diameter $\Omega_{\mathcal{X}}$ of $\mathcal{X}$ is bounded, we can choose $\gamma = \mu/\max\left\{L, \sqrt{\frac{\sum_{i=1}^K {a_i}^2\mathbb{E}[\|\veta_i\|_*^2]}{\Omega_{\mathcal{X}}}}\right\}$, leading to the same (optimal) asymptotic bound as in~\cite{lan2012optimal}. Finally, for bounded-diameter region~\cite{ghadimi2012optimal} provides a step-size policy that can relax the assumption that $K$ is fixed in advance. We expect that a similar policy should also apply in the case of \agdp. The details are omitted. 
\end{remark}

\section{Noise-Error Reduction}\label{sec:noise-reduction}

Based on the results of our analysis from Section~\ref{sec:algorithm}, we now discuss how these results can be used to control the error of \agdp~that is incurred due to the gradient oracle noise. First, we discuss how to prevent error accumulation from Lemma~\ref{lemma:unconstr-noise-bnd}, which is incurred when running a vanilla version of \agdp. The main idea is to take advantage of acceleration until the noise accumulation starts dominating the convergence, and then switch to a slower sequence $\{a_k\}$ for which the error averages out and the algorithm further reduces the mean. Finally, we show how, through another algorithm restart and slow down, the sequence of updates can be made convergent (i.e., the mean error is further reduced at a rate $\sim 1/\sqrt{k}$). 

Observe that the result from Corollary~\ref{cor:convergent-bound} already gives a convergent sequence of updates. However, the choice of parameters in Corollary~\ref{cor:convergent-bound} is fixed and tailored to the global problem properties and worst-case effect of the additive noise. Instead, the strategy of incrementally slowing down the algorithm can take advantage of the more local, fine-grained properties of the objective function. This is confirmed by the numerical experiments provided in Section~\ref{sec:experiments}. 

\subsection{Mean-Error Stabilization}
To take advantage of acceleration at the initial stage and then stabilize the mean error due to noise, we propose the following \rsd~semi-heuristics:


\noindent\vsepfbox{
\begin{minipage}{.98\textwidth}\rsd:  If $\|{\vz_k}\|_2^2 \leq \sum_{i=1}^k {{a_i}^2\mathbb{E}\left[\|\eta_i\|_2^2\right]}$, restart the algorithm taking $\vy_k$ as the initial point and slow down the sequence $\{a_i\}$ to $a_i = \frac{\mu}{L}$, $\forall i \geq 1$. 
\end{minipage}}

The only ``heuristic'' part of \rsd~is deciding when to switch to the slower sequence, as, due to Lemma~\ref{lemma:unconstr-noise-bnd}, slower sequence is guaranteed to lead to a better bound on the approximation error due to noise in the case of unconstrained minimization. Further, switching to a slower, linearly growing sequence $A_k$ is guaranteed to further reduce the error mean, as discussed in the next subsection. 

The intuition behind \rsd~criterion is restarting when ``the signal is drowning in noise''. In particular, $\vz_k$ (the weighted {sum} of the noisy negative gradients) is the only gradient information used in defining all steps of \agdp~and we can interpret it as the ``signal'' that is used to guide algorithm updates. When the gradients are corrupted by noise, ${\vz_k} = - {\sum_{i=1}^k a_i \nabla f(\vx_i)} - {\sum_{i=1}^k a_i \veta_i}$. As the noise is assumed to be independent, the expected energy of the signal-plus-noise is equal to the sum of the energy of the signal and the expected energy of the noise:
\begin{align*}
\left\|{\vz_k}\right\|_2^2 &= \mathbb{E}\left[\left\|{\littlesum_{i=1}^k a_i \nabla f(\vx_i)}\right\|_2^2\right] + \mathbb{E}\left[ \left\|{\littlesum_{i=1}^k a_i \veta_i}\right\|_2^2 \right]\\
&= \left\|{\littlesum_{i=1}^k a_i \nabla f(\vx_i)}\right\|_2^2 + \littlesum_{i=1}^k {{a_i}^2\mathbb{E}\left[\|\eta_i\|_2^2\right]}.
\end{align*}
Hence, when the criterion of \rsd~is satisfied, the energy component due to noise dominates the energy component of the signal in $\vz_k$.

For constrained minimization with a small diameter, \rsd~cannot reduce the theoretical mean of the error due to noise (unless the bound from Lemma~\ref{lemma:unconstr-noise-bnd} dominates the bound from Proposition~\ref{prop:noise-bnd-diam}), as the noise term averages out regardless of the sequence $\{a_i\}$ (Proposition~\ref{prop:noise-bnd-diam}). Nevertheless, a slower, uniform sequence $\{a_i\}$ has lower variance than the accelerated sequence, and can be beneficial in the settings where the accelerated sequence produces high error variance. 

\subsection{Further Mean-Error Reduction}

Quadratically-growing sequence $\{A_i\}$ (or linearly growing sequence $\{a_i\}$) is the fastest-growing sequence which guarantees that $A_k G_k$ is non-increasing in the case of smooth minimization with exact gradients. When we switch to a slower sequence $\{A_i\}$ by invoking \rsd, this creates more slack in making $A_k G_k$ non-increasing in the presence of gradient noise. Hence, \rsd~reduces the mean error and keeps it bounded. However, with \rsd~alone, the mean error cannot converge to zero. To ensure that the error is converging to zero, we can perform an additional \rsd~(\rsd-2), which uses the same criterion for restart, but slows down the sequence $a_k$ to $a_k \sim 1/\sqrt{k}$, as follows.  

\noindent\vsepfbox{
\begin{minipage}{.98\textwidth}\rsd-2:  
If $\|{\vz_k}\|_2^2 \leq \sum_{i=1}^k {{a_i}^2\mathbb{E}\left[\|\eta_i\|_2^2\right]}$, 
restart the algorithm taking $\vy_k$ as the initial point 
and slow down the sequence 
$\{a_i\}$ to 
$a_i = {\mu}/({L}\sqrt{i})$, $\forall i \geq 1$. 
\end{minipage}}

Thus, we have the following Corollary (of Lemma~\ref{lemma:unconstr-noise-bnd}):
\begin{corollary}\label{cor:agdp-convergent}
Let $f:\mathcal{X}\rightarrow \mathbb{R}$ be an $L$-smooth function and let $\vx_0 \in \mathcal{X}$ be an arbitrary initial point. If sequences $\vx_k, \vy_k, \vz_k$ evolve according to (\ref{eq:agdp}) for some $\mu$-strongly convex function $\psi(\cdot)$, where $\veta_k$'s are zero-mean i.i.d. random variables and ${a_k} = \frac{\mu}{L\sqrt{k}}$, then $\forall k \geq 1$:
\begin{equation*}
\begin{aligned}
\mathbb{E}\left[f(\vy_k) - f(\vx_*)\right] \leq \frac{L D_{\psi}(\vx_*, \vx_0)}{\mu \sqrt{k}} + \frac{\log(k+1) \mathbb{E}[\|\veta_1\|_*^2]}{L\sqrt{k}}.
\end{aligned}
\end{equation*}
\end{corollary}
Finally, observe that the factor of $\log(k+1)$ in the bound from Corollary~\ref{cor:agdp-convergent} can be removed if the number of steps $K$ is fixed in advance and $a_k$'s are set to $a_k = \frac{\mu}{L \sqrt{K}}$.

 \section{Different Models of Inexact Oracle}\label{sec:models}

\subsection{Adversarial Models}\label{app:adversarial-models}
There are two main adversarial models of inexact gradient oracles that have been used in the convergence analysis of accelerated methods: the approximate gradient model of~\cite{d2008smooth} and inexact first-order oracle of~\cite{devolder2014first}. The approximate gradient model~\cite{d2008smooth} defines the inexact  oracle by a deterministic perturbation satisfying the following condition for all queries:
\begin{equation*}
\left|\innp{\veta, \vy - \vz}\right| \leq \delta, \; \forall \vy, \vz \in \mathcal{X}.
\end{equation*}
Hence, this model is only applicable to constrained optimization with bounded-diameter domain and bounded (adversarial) additive noise. 
Under these assumptions, \cite{d2008smooth} proves that it is possible to approximate $f(\vx_*)$ up to an error of $\delta$ achieving an accelerated rate.
We can show the same asymptotic bound\footnote{We actually obtain better constants than those in Theorem 2.2 of \cite{d2008smooth}.} by applying the assumption to Equation (\ref{eq:accumulate}) in Proposition~\ref{prop:noise-bnd-diam}. This yields:
$$
G_k \leq  \frac{D_{\psi}(\vx_*, \vx_0)}{A_k} + \delta
$$
whenever $\frac{{a_k}^2}{A_k} \leq \frac{\mu}{L}$ for all $k$. Setting $a_k = \frac{\mu}{L} \cdot \frac{k+1}{2}$ yields $A_k = k^2 + O(k) $, which establishes the accelerated decrease of the first term in the error bound above.

The inexact first-order oracle~\cite{devolder2014first} is a generalization of the model from~\cite{d2008smooth} that defines the inexact oracle by:
\begin{equation*}
f(\vy) \leq f(\vx) + \innp{\nabla f(\vx), \vy - \vx} + \frac{L}{2}\|\vy - \vx\|^2 + \delta.
\end{equation*}
As stated in~\cite{devolder2014first}, this model does not apply to noise-corrupted gradients \emph{per se}, but rather to ``non-smooth and weakly smooth convex problems''. In other words, the model was introduced to characterize the behavior of accelerated methods on objective functions that are non-smooth, but close to smooth. 
Our results agree with those of~\cite{devolder2014first} and lead to the same kind of error accumulation. To see this, observe that we only use the definition of smoothness when bounding $E_k^e$ in Theorem~\ref{thm:noiseless-convergence}. Thus, the error from the inexact oracle would only appear as $E = E_k^e \leq A_k \delta$, leading to:
\begin{equation*}
f(\vy_k) - f(\vx_*) \leq \frac{D_{\psi}(\vx_*, \vx_0)}{A_k} + \frac{\sum_{i=1}^k A_i}{A_k} \cdot \delta,
\end{equation*}
This is exactly the same bound as in Theorem 4 of \cite{devolder2014first}, but we obtain it through a generic algorithm with a simple analysis.

\subsection{Generalized Stochastic Models}\label{app:gen-stochastic-models}

It is possible to generalize the results from Lemma~\ref{lemma:unconstr-noise-bnd} to the noise model from~\cite{lan2012optimal,ghadimi2012optimal}. In such a model, $\veta_k = G(\vx_k, \vxi_k) - \nabla f(\vx_k)$, where $\{\vxi_i\}_{i=1}^k$'s are i.i.d. random vectors, $\mathbb{E}[\veta_k] = \zeros$ and $\mathbb{E}[\|\veta_k\|_*^2]\leq \sigma^2$. Let $\mathcal{F}_k$ denote the natural filtration up to (and including) iteration $k$. Then $\vvh_k = \nabla\psi^*(\vz_k + a_k\veta_k)$ is measurable w.r.t. $\mathcal{F}_{k-1}$ (as $\{\vx_i\}_{i=1}^k$ and $\{\vxi_i\}_{i=1}^{k-1}$ are measurable w.r.t. $\mathcal{F}_{k-1}$). It follows that:
\begin{align}
\mathbb{E}[E_k^{\eta}|\mathcal{F}_{k-1}] &= a_k \mathbb{E}[\innp{\veta_k, \vx_* - \vv_k}|\mathcal{F}_{k-1}] = a_k \mathbb{E}[\innp{\veta_k, \vx_* - \vvh_k}|\mathcal{F}_{k-1}] + a_k \mathbb{E}[\innp{\veta_k, \vvh_k - \vv_k}|\mathcal{F}_{k-1}]\notag\\
&\leq \frac{{a_k}^2}{\mu}\mathbb{E}[\|\veta_k\|_*^2] \leq  \frac{{a_k}^2\sigma^2}{\mu}.\label{eq:cond-noise-error}
\end{align}
Define $\Gamma_k \defeq A_k G_k - \sum_{i=1}^k \frac{{a_k}^2\sigma^2}{\mu}$. Then, using the results from Section~\ref{sec:algorithm} and~\eqref{eq:cond-noise-error}:
\begin{align*}
\mathbb{E}[\Gamma_k - \Gamma_{k-1}|\mathcal{F}_{k-1}] = \mathbb{E}\left[A_kG_k - A_{k-1}G_{k-1} - \frac{{a_k}^2\sigma^2}{\mu}|\mathcal{F}_{k-1}\right] = \mathbb{E}\left[E_k^{\eta} - \frac{{a_k}^2\sigma^2}{\mu}|\mathcal{F}_{k-1}\right]\leq 0,
\end{align*}
i.e., $\Gamma_k$ is a supermartingale. Hence, we can conclude that $\mathbb{E}[\Gamma_k]\leq \mathbb{E}[\Gamma_1],$ implying $\mathbb{E}[G_k] \leq \frac{A_1}{A_k}\mathbb{E}[G_1] + \frac{\sum_{i=2}^k {{a_i}^2\sigma^2} }{\mu A_k}$, and we recover the same bound as in Lemma~\ref{lemma:unconstr-noise-bnd}:
\begin{equation*}
\mathbb{E}[f(\vy_k)] - f(\vx_*) \leq \frac{D_{\psi}(\vx_*, \vx_0)}{A_k} + \frac{\sum_{i=1}^k {{a_i}^2\sigma^2} }{\mu A_k}.
\end{equation*}

\section{\agdp~for Smooth and Strongly Convex Minimization}\label{app:ssc-magdp}

Here we show that~\agdp~can be extended to the setting of smooth and strongly convex minimization. As is customary~\cite{nesterov2013introductory}, in this setting we assume that $\|\cdot\|=\|\cdot\|_2$ so that $f(\cdot)$ is $L$-smooth and $\mu$-strongly convex w.r.t. the $\ell_2$ norm, for $L < \infty$ and $\mu > 0$. To distinguish from the non-strongly-convex case, we refer to \agdp~for smooth and strongly convex minimization as \magdp.


To analyze~\agdp~in this setting, we need to use a stronger lower bound $L_k$, which is constructed by the same arguments as before, but now using strong convexity instead of regular convexity. Such a construction gives:
\begin{equation}\label{eq:smooth-sc-lb}
\begin{aligned}
L_k =& \frac{\sum_{i=1}^k a_i f(\vx_i) + \min_{\vu \in \mathcal{X}}m_k(\vu) - \sum_{i=1}^k a_i\innp{\veta_i, \vx_* - \vx_i} - D_{\psi}(\vx_*, \vx_0)}{A_k},
\end{aligned}
\end{equation}
where
\begin{equation*}
\begin{aligned}
m_k(\vu) =& \frac{\sum_{i=1}^k a_i \left(\innp{\nabla f(\vx_i), \vu - \vx_i} + \frac{\mu}{2}\|\vu - \vx_i\|^2\right) + D_{\psi}(\vu, \vx_0)}{A_k}.
\end{aligned}
\end{equation*}
While it suffices to have $\psi$ be an arbitrary function that is strongly convex w.r.t. the $\|\cdot\|_2$, for simplicity, we take $\psi(\vx) = \frac{\mu_0}{2}\|\vx\|^2$, where $\mu_0$ will be specified later.

For $\theta_k = \frac{a_k}{A_k}$, the algorithm can now be stated as follows:
\begin{equation}\label{eq:magdp}\tag{\magdp}
\begin{gathered}
\vv_k = \argmin_{\vu \in \mathcal{X}}m_k(\vu),\\
\vx_k = \frac{1}{1+\theta_k}\vy_{k-1} + \frac{\theta_k}{1+\theta_k}\vv_{k-1},\\
\vy_{k} = (1-\theta_k)\vy_{k-1} + \theta_k \vv_k,
\end{gathered}
\end{equation}
where, $\vx_1 = \vx_0 = \vy_0 = \vv_0$ is an arbitrary initial point from $\mathcal{X}$.

As before, the main convergence argument is to show that $A_k G_k \leq A_{k-1}G_{k-1}$ and combine it with the bound on the initial gap $G_1$. We start with bounding the initial gap, as follows.

\begin{proposition}\label{prop:init-gap-ssc}
If $\psi(\vx) = \frac{\mu_0}{2}\|\vx\|^2$, where $\mu_0 = a_1(L-\mu)$, then $A_1G_1 \leq \frac{A_1(L-\mu)}{2}\|\vx_* - \vx_0\|^2 + E_1^{\eta}$, where $E_1^{\eta} = a_1 \innp{\veta_1, \vx_* - \vv_1}$.
\end{proposition}
\begin{proof}
As $\vx_1 = \vx_0$, the initial lower bound is:
\begin{align*}
A_1 L_1 =& a_1 f(\vx_1) + a_1 \innp{\nabla f(\vx_1), \vv_1 -\vx_1}+ \left(\frac{a_1\mu}{2} + \frac{\mu_0}{2}\right)\|\vv_1 - \vx_1\|^2 - \frac{\mu_0}{2}\|\vx_* - \vx_0\|^2 - a_1 \innp{\veta_1, \vx_* - \vx_1}.
\end{align*}
As $a_1 = A_1$, it follows that $\vy_1 = \vv_1$, and hence:
\begin{align*}
A_1 U_1 =& a_1 f(\vv_1)\\
\leq & a_1f(\vx_1) + a_1 \innp{\nabla f(\vx_1), \vv_1 - \vx_1}+ \frac{a_1L}{2}\|\vv_1 - \vx_1\|^2\\
=& a_1f(\vx_1) + a_1 \innp{\tnabla f(\vx_1), \vv_1 - \vx_1}+ \frac{a_1L}{2}\|\vv_1 - \vx_1\|^2 - a_1 \innp{\veta_1, \vv_1 - \vx_1},
\end{align*}
where the inequality is by the smoothness of $f(\cdot)$. 
Combining the bounds on the initial upper and lower bounds, it follows:
\begin{align*}
A_1G_1 &\leq \frac{\mu_0}{2}\|\vx_* - \vx_0\|^2 + \frac{a_1L - a_1\mu - \mu_0}{2}\|\vv_1 - \vx_1\|^2 + a_1 \innp{\veta_1, \vx_* - \vv_1}\\
&\leq \frac{a_1(L-\mu)}{2}\|\vx_* - \vx_0\|^2 + E_1^{\eta},
\end{align*}
as, by the initial assumption, $\mu_0 = a_1(\mu - L)$ and $E_1^{\eta} = a_1 \innp{\veta_1, \vx_* - \vv_1}$.
\end{proof}

To bound the change in the lower bound, it is useful to first bound $m_k(\vv_k) - m_{k-1}(\vv_{k-1})$, as in the following technical proposition.

\begin{proposition}\label{prop:change-in-min-ssc}
Let $\psi(\vx) = \frac{a_1(L-\mu)}{2}\|\vx\|^2$. Then: 
\begin{equation*}
\begin{aligned}
m_{k}(\vv_k)\geq & m_{k-1}(\vv_{k-1}) + a_k \innp{\tnabla f(\vx_k), \vv_k - \vx_k} + \frac{A_{k-1}\mu}{2}\|\vv_k - \vv_{k-1}\|^2
+ \frac{a_k \mu}{2}\|\vv_k - \vx_k\|^2.
\end{aligned}
\end{equation*}
\end{proposition}
\begin{proof}
Observe that, by the definition of $m_k(\cdot)$, $m_k(\vv_k) = m_{k-1}(\vv_k) + a_k \innp{\tnabla f(\vx_k), \vv_k - \vx_k} + a_k \frac{\mu}{2}\|\vv_k - \vx_k\|.$ 

The rest of the proof bounds $m_{k-1}(\vv_k) - m_{k-1}(\vv_{k-1})$. Observe that, as $\vv_{k-1} = \argmin_{\vu \in \mathcal{X}}m_{k-1}(\vu)$, it must be $\innp{\nabla m_{k-1}(\vv_{k-1}), \vu - \vv_{k-1}} \geq 0$, $\forall \vu \in \mathcal{X}$. As Bregman divergence is blind to linear terms:
\begin{align*}
m_{k-1}(\vv_k) - m_{k-1}(\vv_{k-1}) &=  \innp{\nabla m_{k-1}(\vv_{k-1}), \vu - \vv_{k-1}} + D_{m_{k-1}}(\vv_{k}, \vv_{k-1})\\
&\geq D_{m_{k-1}}(\vv_{k}, \vv_{k-1})\\
&= \frac{A_{k-1}\mu}{2}\|\vv_k - \vv_{k-1}\|^2 + \frac{a_{1}(L-\mu)}{2}\|\vv_k - \vv_{k-1}\|^2.
\end{align*}
The rest of the proof is by $\frac{a_{1}(L-\mu)}{2}\|\vv_k - \vv_{k-1}\|^2 \geq 0$.
\end{proof}

We are now ready to move to the main part of the convergence argument, namely, to show that $A_k G_k \leq A_{k-1}G_{k-1}$ for a certain choice of $a_k$.

\begin{lemma}\label{lema:ssc-gap-reduction}
Let $\psi(\vx) = \frac{a_1(L-\mu)}{2}\|\vx\|^2$ and $0<a_k \leq A_k \sqrt{\frac{\mu}{L}}$. Then: $A_k G_k \leq A_{k-1}G_{k-1} + E_k^{\eta}$, where $E_k^{\eta} = a_k \innp{\veta_k, \vx_* - \vv_k}$.
\end{lemma}
\begin{proof}
As $U_k = f(\vy_k)$, we have that:
\begin{equation}\label{eq:ssc-ub-change}
A_{k}U_k - A_{k-1}U_{k-1} \leq A_k f(\vy_k) - A_{k-1}f(\vy_{k-1}).
\end{equation}

Using Proposition~\ref{prop:change-in-min-ssc}, the change in the lower bound is:
\begin{equation}\label{eq:ssc-lb-change}
\begin{aligned}
A_kL_k - A_{k-1}L_{k-1}
\geq & a_k f(\vx_k) +a_k\innp{\tnabla f(\vx_k), \vv_k - \vx_k} + \frac{A_{k-1}\mu}{2}\|\vv_k - \vv_{k-1}\|^2 + \frac{a_k\mu}{2}\|\vv_k - \vx_k\|^2\\
&+ a_k \innp{\veta_k, \vx_* - \vx_k}.
\end{aligned}
\end{equation}

Denote $\vw_{k} = \frac{A_{k-1}}{A_k}\vv_{k-1} + \frac{a_k}{A_k}\vx_{k}$. By Jensen's Inequality:
\begin{equation}\label{eq:ssc-jensen}
\begin{aligned}
\frac{A_{k-1}\mu}{2}\|\vv_k - \vv_{k-1}\|^2 + \frac{a_k\mu}{2}\|\vv_k - \vx_k\|^2 \geq \frac{\mu A_k}{2}\|\vv_k - \vw_k\|^2.
\end{aligned}
\end{equation}
Write $a_k\innp{\tnabla f(\vx_k), \vv_k - \vx_k}$ as:
\begin{equation}\label{eq:ssc-innp}
\begin{aligned}
 a_k\innp{\tnabla f(\vx_k), \vv_k - \vx_k} = a_k\innp{\nabla f(\vx_{k}), \vv_k - \vw_k} + a_k \innp{\nabla f(\vx_k), \frac{A_{k-1}}{A_k}(\vv_{k-1}-\vx_k)} + a_k \innp{\veta_k, \vv_k - \vx_k}.
\end{aligned}
\end{equation}
As $\frac{a_k}{A_k} \leq \sqrt{\frac{\mu}{L}}$ and by smoothness of $f(\cdot):$
\begin{align}
a_k\innp{\nabla f(\vx_{k}), \vv_k - \vw_k} + \frac{\mu A_k}{2}\|\vv_k - \vw_k\|^2
&\geq A_k \bigg(\innp{\nabla f(\vx_k), \Big(\vx_k + \frac{a_k}{A_k} (\vv_k - \vw_k)\Big)-\vx_k} + \frac{L}{2}\|\frac{a_k}{A_k}(\vv_k - \vw_k)\|^2\bigg)\notag\\
&\geq A_k\left(f\left(\vx_k + \frac{a_k}{A_k} (\vv_k - \vw_k)\right) - f(\vx_k)\right).\label{eq:ssc_sc-wk-v-k}
\end{align}

Combining (\ref{eq:ssc-lb-change})-(\ref{eq:ssc_sc-wk-v-k}), we have the following bound for the change in the lower bound:
\begin{equation*}
\begin{aligned}
A_k L_k - A_{k-1}L_{k-1} \geq & A_k f\left(\vx_k + \frac{a_k}{A_k} (\vv_k - \vw_k)\right) - A_{k-1}f(\vx_k) + a_k \innp{\nabla f(\vx_k), \frac{A_{k-1}}{A_k}(\vv_{k-1}-\vx_k)}\\
&+ a_k \innp{\veta_k, \vx_* - \vv_k}.
\end{aligned}
\end{equation*}
Using the definition of $\vw_k$, $\theta_k = \frac{a_k}{A_k},$ and (\ref{eq:magdp}), it is not hard to verify that $\vy_k = \vx_k + \frac{a_k}{A_k} (\vv_k - \vw_k)$ and $\frac{a_k}{A_k}(\vv_{k-1}-\vx_k) = \vx_k - \vy_{k-1}$, which, using the convexity of $f(\cdot)$, gives:
\begin{equation}\label{eq:ssc-lb-change-final}
A_{k}L_k - A_{k-1}L_{k-1} \geq A_k f(\vy_k) - A_{k-1}f(\vy_{k-1}) - a_k \innp{\veta_k, \vx_* - \vv_k}.
\end{equation}
Combining (\ref{eq:ssc-lb-change-final}) and (\ref{eq:ssc-ub-change}), the proof follows.
\end{proof}

\begin{theorem}\label{thm:magdp-convergence}
Let $\psi(\vx) = \frac{L-\mu}{2}\|\vx\|^2$, $a_1 = 1$, $\frac{a_i}{A_i} = \gamma_i \leq \sqrt{\frac{\mu}{L}}$ for $i \geq 2$, and let $\vy_k, \vx_k$ evolve according to (\ref{eq:magdp}). Then, $\forall k \geq 1$:
\begin{align*}
f(\vy_k) - f(\vx_*) &= \frac{A_1}{A_k}\cdot\frac{(L-\mu)\|\vx_* - \vx_0\|^2}{2} + \frac{\sum_{i=1}^k a_i \innp{\veta_i, \vx_* - \vv_i}}{A_k}\\
&\leq \left(\Pi_{i=1}^k\left(1- \gamma_i\right)\right)\frac{(L-\mu)\|\vx_* - \vx_0\|^2}{2} + \frac{\sum_{i=1}^k a_i \innp{\veta_i, \vx_* - \vv_i}}{A_k}.
\end{align*}
\end{theorem}
\begin{proof}
Applying Lemma~\ref{lema:ssc-gap-reduction}, it follows that $G_k \leq \frac{A_1 G_1}{A_k} + \frac{\sum_{i=1}^k E_i^{\eta}}{A_k} = \frac{A_1}{A_2}\cdot \frac{A_2}{A_3}\cdot\dots\cdot \frac{A_{k-1}}{A_k}G_1 + \frac{\sum_{i=1}^k E_i^{\eta}}{A_k}$. As $\frac{A_{i-1}}{A_i} = 1 - \frac{a_i}{A_i} = 1 - \gamma_i$, we have $G_k \leq \left(\Pi_{i=1}^k\left(1- \gamma_i\right)\right)G_1 + \frac{\sum_{i=1}^k E_i^{\eta}}{A_k}$. The rest of the proof is by applying Proposition~\ref{prop:init-gap-ssc} and using that $f(\vy_k) - f(\vx_*)\leq G_k$. 
\end{proof}
Using the same arguments for bounding the noise term as in the case of smooth minimization (Section~\ref{sec:algorithm}), we have the following corollary.
\begin{corollary}[of Theorem~\ref{thm:magdp-convergence}]\label{cor:ssc-noise-bounds}
If $\veta_i = \zeros$ (the noiseless gradient case), setting $\gamma_i = \sqrt{\frac{\mu}{L}}$, we recover the standard convergence result for accelerated smooth and strongly convex minimization:
$$
f(\vy_k) - f(\vx_*) \leq \left(1- \sqrt{\frac{\mu}{L}}\right)^k\frac{L-\mu}{2}\|\vx_* - \vx_0\|^2.
$$
If $\mathbb{E}[\|\veta_i\|]\leq M_i$ and $\max_{\vu \in \mathcal{X}}\|\vx_* - \vu\|\leq R_{\vx_*}$, then setting $\gamma_i = \sqrt{\frac{\mu}{L}}$:
$$
\mathbb{E}[f(\vy_k) - f(\vx_*)] \leq \left(1- \sqrt{\frac{\mu}{L}}\right)^k\frac{L-\mu}{2}\|\vx_* - \vx_0\|^2 + \frac{R_{\vx_*}\sum_{i=1}^k a_i M_i}{A_k}.
$$
If $\veta_i$'s are zero-mean and independent and $\mathbb{E}[\|\veta_i\|^2]\leq \sigma^2$, then:
\begin{align*}
\mathbb{E}[f(\vy_k) - f(\vx_*)] &\leq \left(\Pi_{i=1}^k\left(1- \gamma_i\right)\right)\frac{L-\mu}{2}\|\vx_* - \vx_0\|^2 + \frac{\sigma^2\sum_{i=1}^k \frac{{a_i}^2}{A_i}}{\mu A_k}
\end{align*}
In particular, setting:
\begin{itemize}
\item $\frac{a_i}{A_i}=\gamma_i = \sqrt{\frac{\mu}{L}},$ 
$$
\mathbb{E}[f(\vy_k) - f(\vx_*)] \leq  \left(1- \sqrt{\frac{\mu}{L}}\right)^k\frac{L-\mu}{2}\|\vx_* - \vx_0\|^2 + \frac{\sigma^2}{\sqrt{\mu L} }
$$
\item $a_i = i^p$ for $p \in \mathbb{Z}_+,$
$$
\mathbb{E}[f(\vy_k) - f(\vx_*)] = O \left(\frac{p+1}{k^{p+1}}\cdot\frac{(L-\mu)\|\vx_* - \vx_0\|^2}{2} + \frac{(p+1)^2}{p k}\cdot\frac{\sigma^2}{\mu}\right).
$$
\end{itemize}
\end{corollary}
\begin{proof}
The bounds for $\veta_i = \zeros$ and for $\mathbb{E}[\|\veta_i\|]\leq M_i$ and $\max_{\vu \in \mathcal{X}}\|\vx_* - \vu\|\leq R_{\vx_*}$ are straightforward. 

Assume that $\veta_i$'s are zero-mean and independent and denote $\psi_k(\vx) = \sum_{i=1}^k a_i \frac{\mu}{2}\|\vx - \vx_i\|^2 + \frac{\mu_0}{2}\|\vx - \vx_0\|^2$. Observe that the strong convexity parameter of $\psi_k$ is $\mu A_k + \mu_0>\mu A_k.$ From Fact~\ref{fact:danskin}, $\vv_k = \nabla \psi^*_k(\vz_k)$. Similarly as for the case of smooth minimization, let $\vvh_k = \nabla \psi^*(\vz_k + a_k \veta_k)$. Then $\vvh_k$ is independent of $\veta_k,$ and, using Fact~\ref{fact:smoothness-of-cvx-conj}, we have:
\begin{align*}
\mathbb{E}[a_k\innp{\veta_k, \vx_* - \vv_k}] &= \mathbb{E}[a_k \innp{\veta_k, \vx_* - \vvh_k}] + \mathbb{E}[a_k \innp{\veta_k, \vvh_k - \vv_k}]\\
&\leq \frac{{a_k}^2}{\mu A_k}\|\veta_k\|^2.
\end{align*}
Combining with Theorem~\ref{thm:magdp-convergence}, we get $\mathbb{E}[f(\vy_k) - f(\vx_*)] \leq \left(\Pi_{i=1}^k\left(1- \gamma_i\right)\right)\frac{L-\mu}{2}\|\vx_* - \vx_0\|^2 + \frac{\sigma^2\sum_{i=1}^k \frac{{a_i}^2}{A_i}}{\mu A_k}.$ The rest of the proof follows by plugging in particular choices of $a_i$.
\end{proof}

Let us make a few more remarks here. When $\veta_i$'s are zero-mean, independent, and $\mathbb{E}[\|\veta_i\|^2]\leq \sigma^2,$ even the vanilla version of \magdp~ does not accumulate noise (the noise averages out). Under the same assumptions and when $a_i = i,$ we recover the asymptotic bound from~\cite{ghadimi2012optimal}.\footnote{Note that the independence of $\veta_i$'s is a stronger assumption than used in~\cite{ghadimi2012optimal}. Nevertheless, we can obtain the same bounds as stated in Corollary~\ref{cor:ssc-noise-bounds} for the same assumptions as in~\cite{ghadimi2012optimal} using the same arguments as for the case of smooth minimization (see Section~\ref{app:gen-stochastic-models}).} More generally, $a_i = i^p$ for any constant integer $p$ gives a convergence bound for which the deterministic term vanishes at rate $1/k^{p+1}$ while the noise term vanishes at rate $1/k.$ When $p = \log(k)$  for a fixed number of iterations $k$ of \magdp, we get 
$
\mathbb{E}[f(\vy_k) - f(\vx_*)] = O \left(\frac{\log(k)}{k^{\log(k)}}\cdot\frac{(L-\mu)\|\vx_* - \vx_0\|^2}{2} + \frac{\log(k)}{k}\cdot\frac{\sigma^2}{\mu}\right),$ i.e., the deterministic term (independent of noise) decreases super-polynomially with the iteration count, while the noise term decreases at rate $\log(k)/k$. This is a much stronger bound than the one from~\cite{ghadimi2012optimal} and closer to the theoretical lower bound $\Omega \left(\left(1- \sqrt{\frac{\mu}{L}}\right)^k\cdot\frac{(L-\mu)\|\vx_* - \vx_0\|^2}{2} + \cdot\frac{\sigma^2}{\mu k}\right)$ from~\cite{nemirovskii1983problem}.

Note that in the setting of constrained (bounded-diameter) minimization, \cite{ghadimi2013optimal} obtained the optimal convergence bound $O \left(\left(1- \sqrt{\frac{\mu}{L}}\right)^k\cdot\frac{(L-\mu)\|\vx_* - \vx_0\|^2}{2} + \cdot\frac{\sigma^2}{\mu k}\right)$ by coupling the algorithm from~\cite{ghadimi2012optimal} with a domain-shrinking procedure resulting in a multi-stage algorithm. We expect it is possible to obtain a similar result for \magdp~by coupling it with the domain-shrinking from~\cite{ghadimi2013optimal}.

\section{Numerical Experiments}\label{sec:experiments}

 \begin{figure*}[ht!]
 \centering
  \vspace{-10pt}
 \subfigure[$\sigma_{\eta}=0$]{\includegraphics[width=0.24\textwidth]{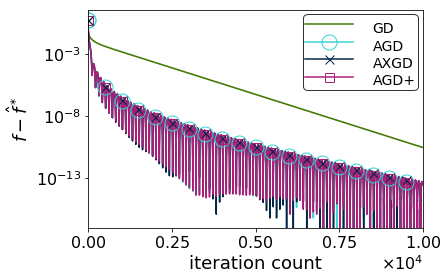}\label{fig:cyc_uc_0}}\vspace{-10pt}
 \subfigure[$\sigma_{\eta}=10^{-5}$]{\includegraphics[width=0.24\textwidth]{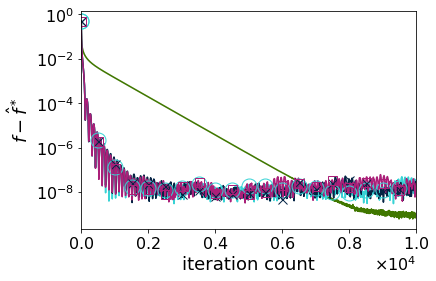}\label{fig:cyc_uc_1e-5}}
 \subfigure[$\sigma_{\eta}=10^{-3}$]{\includegraphics[width=0.24\textwidth]{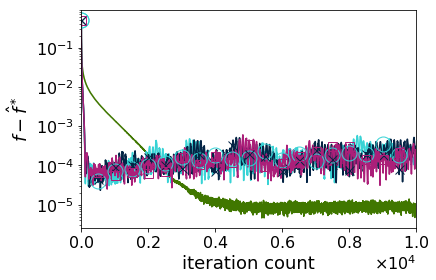}\label{fig:cyc_uc_1e-3}}
 \subfigure[$\sigma_{\eta}=10^{-1}$]{\includegraphics[width=0.24\textwidth]{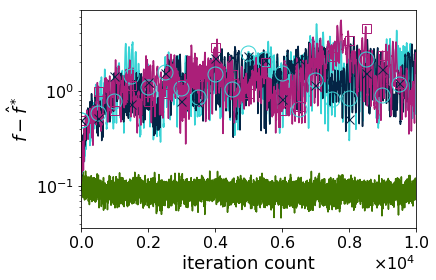}\label{fig:cyc_uc_1e-1}}
 \subfigure[$\sigma_{\eta}=0$]{\includegraphics[width=0.24\textwidth]{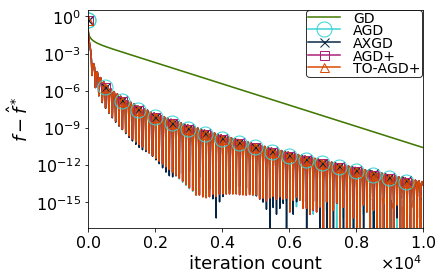}\label{fig:r_cyc_uc_0_sr}} \vspace{-10pt}
 \subfigure[$\sigma_{\eta}=10^{-5}$]{\includegraphics[width=0.24\textwidth]{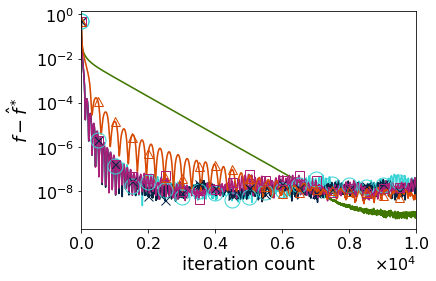}\label{fig:r_cyc_uc_1e-5_sr}}
 \subfigure[$\sigma_{\eta}=10^{-3}$]{\includegraphics[width=0.24\textwidth]{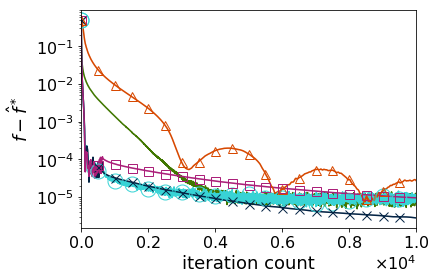}\label{fig:r_cyc_uc_1e-3_sr}}
 \subfigure[$\sigma_{\eta}=10^{-1}$]{\includegraphics[width=0.24\textwidth]{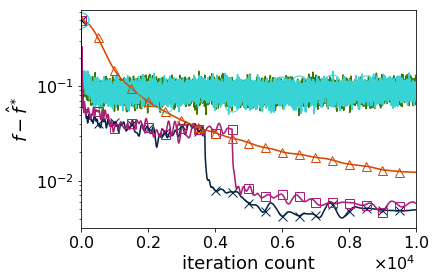}\label{fig:r_cyc_uc_1e-1_sr}}
\subfigure[$\sigma_{\eta}=0$]{\includegraphics[width=0.24\textwidth]{r_cycle_unconstrained_0.png}\label{fig:r_cyc_uc_0}}\vspace{-10pt}
 \subfigure[$\sigma_{\eta}=10^{-5}$]{\includegraphics[width=0.24\textwidth]{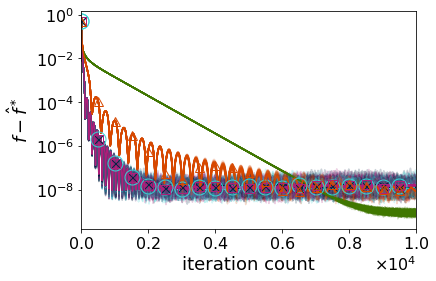}\label{fig:r_cyc_uc_1e-5}}
 \subfigure[$\sigma_{\eta}=10^{-3}$]{\includegraphics[width=0.24\textwidth]{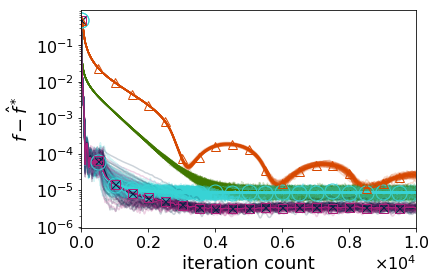}\label{fig:r_cyc_uc_1e-3}}
 \subfigure[$\sigma_{\eta}=10^{-1}$]{\includegraphics[width=0.24\textwidth]{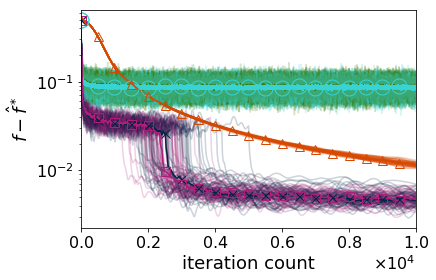}\label{fig:r_cyc_uc_1e-1}}
\caption{Performance of gradient descent (\gd) and accelerated algorithms (\agd, \axgd, \agdp) on a hard instance for unconstrained smooth minimization for $\veta_k \sim \mathcal{N}(\zeros, \sigma_\eta I)$ over $\mathbb{R}^n$:  \protect\subref{fig:cyc_uc_0}-\protect\subref{fig:cyc_uc_1e-1} sample run; \protect\subref{fig:r_cyc_uc_0_sr}-\protect\subref{fig:r_cyc_uc_1e-1_sr} sample run with \textsc{to}-\agdp, \rsd~and~\rsd-2; \protect\subref{fig:r_cyc_uc_0}-\protect\subref{fig:r_cyc_uc_1e-1} median over 50 runs with \textsc{to}-\agdp, \rsd~and \rsd-2.}
\label{fig:test1}\vspace{-10pt}
 \end{figure*}

To illustrate the results, we consider two main problems: a hard instance for smooth minimization (see, e.g.,~\cite{nesterov2013introductory}) and regression problems on Epileptic Seizure Recognition Dataset~\cite{andrzejak2001indications} obtained from the UCI Machine Learning Repository~\cite{Lichman:2013}. All of the problems are in the domain of smooth minimization. The experiments for the case of smooth and strongly convex minimization are deferred to the future version of the paper. 

In all the experiments, we used standard Python libraries to solve the considered problems to high accuracy. The resulting function value is denoted by $\hat{f}^*$ in the figures. In all the problems, we used $\psi(\vx) = \frac{L}{2}\|\vx\|_2^2$ as the regularizer. For constrained problems, we implemented projected gradient descent as the ``\gd'' algorithm. 

In the graphs, \textsc{to-}\agdp~denotes the ``theoretically optimal'' version of \agdp; namely, it corresponds to \agdp~with step sizes chosen according to Corollary~\ref{cor:convergent-bound} and Remark~\ref{remark:convergent-bound}. In all the experiments, we compare the different accelerated algorithms (\agdp, \agd, \axgd) and the non-accelerated \gd~under i.i.d. additive gradient noise $\veta_i \sim \mathcal{N}(\zeros, \sigma_\eta I)$. 

\subsection{``Hard'' Instance for Smooth Minimization}

\paragraph{Unconstrained Minimization}
To understand the worst-case performance of \agdp, we first compare it to Nesterov's~\agd~and~\cite{AXGD}'s \axgd. The instance is an unconstrained minimization problem, where $f(\vx) = \frac{1}{2}\innp{\mA\vx, \vx}-\innp{\vb, \vx}$, $\mA$ is the graph Laplacian of a cycle\footnote{Namely, the difference of a tridiagonal square matrix $\mathbf{C}$ with 1's on the main diagonal and -1's on the remaining diagonals, and matrix $\mathbf{B}$, which is zero everywhere except for $B_{1n} = B_{n1} = 1$.}, $b_1 = -b_n = 1$ and vector $\vb$ is zero elsewhere. The initial point $\vx_0$ is an all-zeros vector. The dimension of the problem is $n = 100$.

The performance of \agdp, \agd, and \axgd~together with the performance of the slower, unaccelerated \gd~on the described worst-case instance is shown in Fig.~\ref{fig:cyc_uc_0}-\ref{fig:cyc_uc_1e-1}, for the exact gradient oracle (Fig.~\ref{fig:cyc_uc_0}) and noise-corrupted gradient oracle with i.i.d. $\veta_i \sim \mathcal{N}(\zeros, \sigma_\eta I)$ (Fig.~\ref{fig:cyc_uc_1e-5}-\ref{fig:cyc_uc_1e-1}). We repeated the same experiments when the parameters for \agdp~are chosen according to Corollary~\ref{cor:convergent-bound} (denoted as \textsc{to-}\agdp) and when  \rsd~and \rsd-2~are employed (Fig.~\ref{fig:r_cyc_uc_0_sr}-\ref{fig:r_cyc_uc_1e-1}).

Without restart and slow-down, all accelerated algorithms perform similarly. In particular, as the noise standard deviation $\sigma_{\eta}$ is increased, the mean and the variance of the approximation error of all accelerated algorithms increases and the noise appears to be accumulating (see, e.g., Fig.~\ref{fig:cyc_uc_1e-1}). On the other hand, \gd~generally converges to an approximation error with lower mean and variance, at the expense of converging at a slower $1/k$ rate. 

When restart and slow-down are used, in the noiseless case (Fig.~\ref{fig:r_cyc_uc_0_sr} and~\ref{fig:r_cyc_uc_0}), there is no difference compared to the vanilla case (Fig.~\ref{fig:cyc_uc_0}), which is what we want -- there is no need to slow down the accelerated algorithms unless their performance is compromised by noise. In the low-noise scenario (Fig.~\ref{fig:r_cyc_uc_1e-5_sr}, \ref{fig:r_cyc_uc_1e-5}), \rsd~does not change the performance of the algorithms in a noticeable way, although, in that case, the performance degradation due to noise is low. As the noise becomes higher (Fig.~\ref{fig:r_cyc_uc_1e-3_sr}, \ref{fig:r_cyc_uc_1e-3}, \ref{fig:r_cyc_uc_1e-5_sr}, \ref{fig:r_cyc_uc_1e-1}),  restart and slow-down noticeably stabilizes all accelerated algorithms, reducing both their mean and their variance. Further, restart and slow-down generally outperforms the ``theoretically optimal'' \agdp~(\textsc{to-}\agdp). 

\begin{figure*}
\centering
\subfigure[$\sigma_{\eta}=0$]{\includegraphics[width=0.24\textwidth]{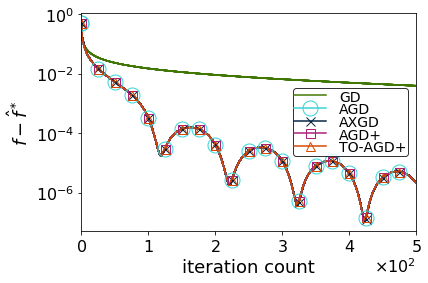}\label{fig:r_cyc_uc_0_500}}
 \subfigure[$\sigma_{\eta}=10^{-5}$]{\includegraphics[width=0.24\textwidth]{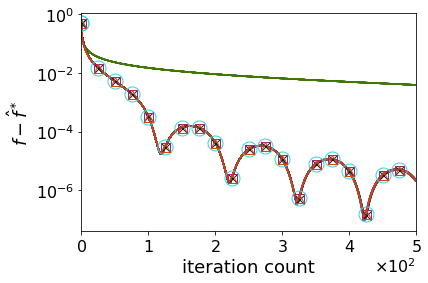}\label{fig:r_cyc_uc_1e-5_500}}
 \subfigure[$\sigma_{\eta}=10^{-3}$]{\includegraphics[width=0.24\textwidth]{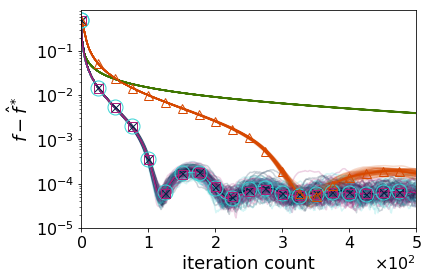}\label{fig:r_cyc_uc_1e-3_500}}
 \subfigure[$\sigma_{\eta}=10^{-1}$]{\includegraphics[width=0.24\textwidth]{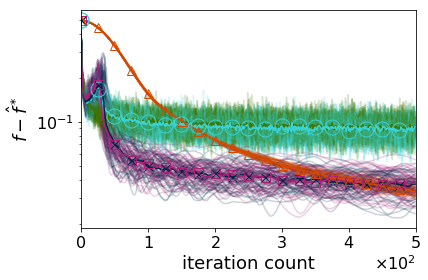}\label{fig:r_cyc_uc_1e-1_500}}
 \subfigure[$\sigma_{\eta}=0$]{\includegraphics[width=0.24\textwidth]{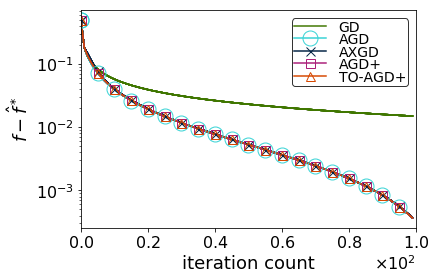}\label{fig:r_cyc_uc_0_100}}
 \subfigure[$\sigma_{\eta}=10^{-5}$]{\includegraphics[width=0.24\textwidth]{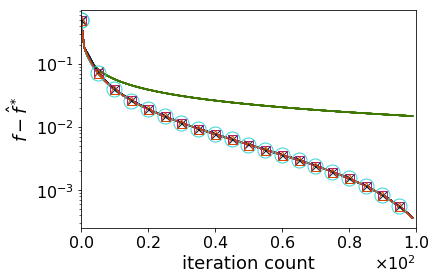}\label{fig:r_cyc_uc_1e-5_100}}
 \subfigure[$\sigma_{\eta}=10^{-3}$]{\includegraphics[width=0.24\textwidth]{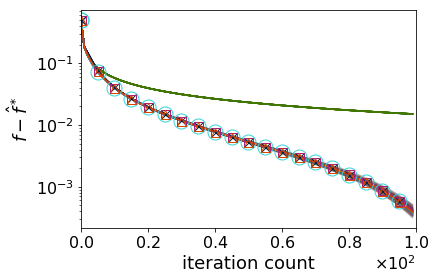}\label{fig:r_cyc_uc_1e-3_100}}
 \subfigure[$\sigma_{\eta}=10^{-1}$]{\includegraphics[width=0.24\textwidth]{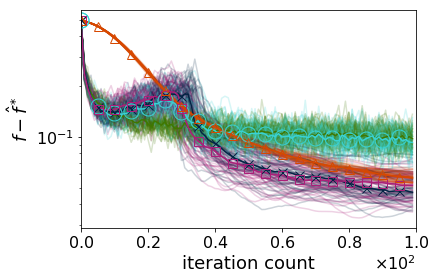}\label{fig:r_cyc_uc_1e-1_100}}
\caption{Median performance of gradient descent (\gd) and accelerated algorithms (\agd, \axgd, \agdp~with restart and slow-down semi-heuristics and \textsc{to-}\agdp) over 50 repeated runs on a hard instance for unconstrained smooth minimization for $\veta_k \sim \mathcal{N}(\zeros, \sigma_\eta I)$ over $\mathbb{R}^n$ and for:  \protect\subref{fig:r_cyc_uc_0_500}-\protect\subref{fig:r_cyc_uc_1e-1_500} 500 iterations; \protect\subref{fig:r_cyc_uc_0_100}-\protect\subref{fig:r_cyc_uc_1e-1_100} 100 iterations.}
\label{fig:uc-cyc-it-small}
 \end{figure*}

Since the step sizes of \textsc{to-}\agdp~depend on the number of iterations, one may hope that \textsc{to-}\agdp~could outperform other accelerated algorithms for a smaller number of iterations. However, this is not true -- for a smaller number of iterations, in the best case \textsc{to-}\agdp~matches the performance of other algorithms with restart and slow-down. When all algorithms have the same performance, they all essentially run their vanilla versions -- without restart and slow down and for their standard accelerated step sizes. This is illustrated in Fig.~\ref{fig:uc-cyc-it-small}.

\begin{figure*}[t!]
\centering
 \subfigure[$\sigma_{\eta}=0$]{\includegraphics[width=0.24\textwidth]{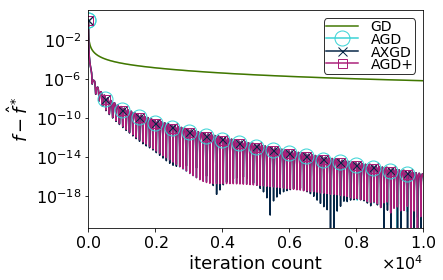}\label{fig:cyc_s_0}}
 \subfigure[$\sigma_{\eta}=10^{-5}$]{\includegraphics[width=0.24\textwidth]{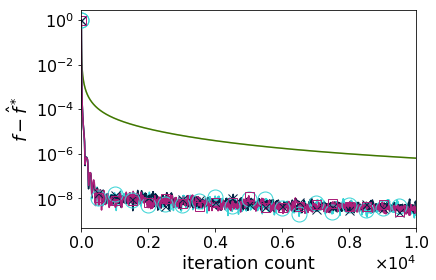}\label{fig:cyc_s_1e-5}}
 \subfigure[$\sigma_{\eta}=10^{-3}$]{\includegraphics[width=0.24\textwidth]{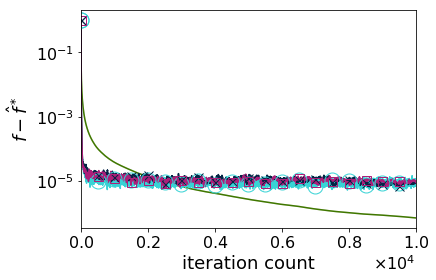}\label{fig:cyc_s_1e-3}}
 \subfigure[$\sigma_{\eta}=10^{-1}$]{\includegraphics[width=0.24\textwidth]{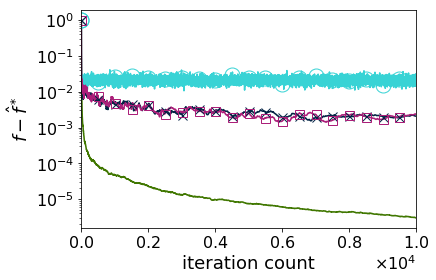}\label{fig:cyc_s_1e-1}}
 \subfigure[$\sigma_{\eta}=0$]{\includegraphics[width=0.24\textwidth]{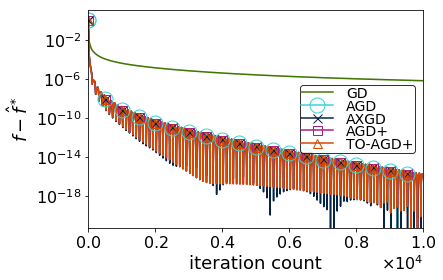}\label{fig:r_cyc_s_0_sr}}
 \subfigure[$\sigma_{\eta}=10^{-5}$]{\includegraphics[width=0.24\textwidth]{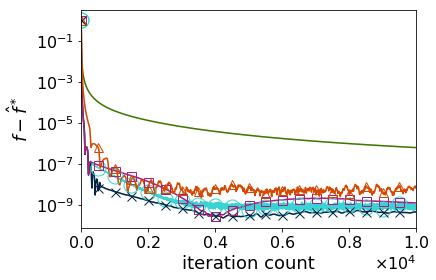}\label{fig:r_cyc_s_1e-5_sr}}
 \subfigure[$\sigma_{\eta}=10^{-3}$]{\includegraphics[width=0.24\textwidth]{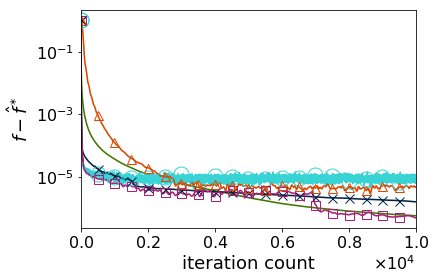}\label{fig:r_cyc_s_1e-3_sr}}
 \subfigure[$\sigma_{\eta}=10^{-1}$]{\includegraphics[width=0.24\textwidth]{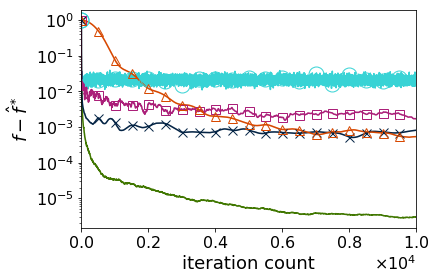}\label{fig:r_cyc_s_1e-1_sr}}
 \subfigure[$\sigma_{\eta}=0$]{\includegraphics[width=0.24\textwidth]{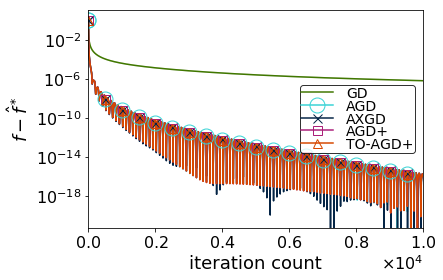}\label{fig:r_cyc_s_0}}
 \subfigure[$\sigma_{\eta}=10^{-5}$]{\includegraphics[width=0.24\textwidth]{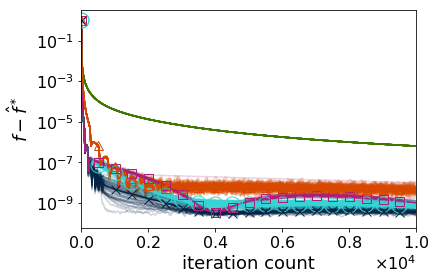}\label{fig:r_cyc_s_1e-5}}
 \subfigure[$\sigma_{\eta}=10^{-3}$]{\includegraphics[width=0.24\textwidth]{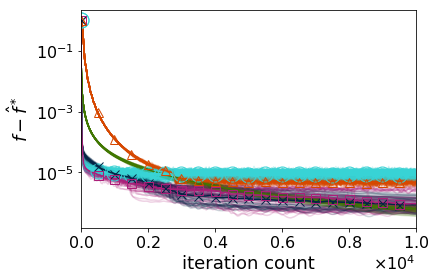}\label{fig:r_cyc_s_1e-3}}
 \subfigure[$\sigma_{\eta}=10^{-1}$]{\includegraphics[width=0.24\textwidth]{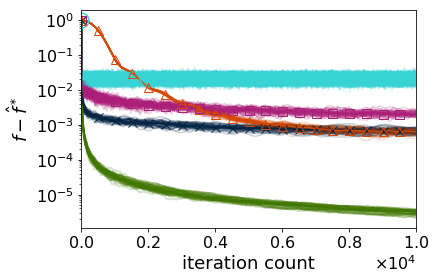}\label{fig:r_cyc_s_1e-1}}
\caption{Performance of gradient descent (\gd) and accelerated algorithms (\agd, \axgd, \agdp) on a hard instance for unconstrained smooth minimization for $\veta_k \sim \mathcal{N}(\zeros, \sigma_\eta I)$ and over probability simplex  \protect\subref{fig:cyc_s_0}-\protect\subref{fig:cyc_s_1e-1} without and \protect\subref{fig:r_cyc_s_0}-\protect\subref{fig:r_cyc_s_1e-1} with \rsd~and \rsd-2.}
\label{fig:test1c}
\end{figure*}

\paragraph{``Hard Instance'' over Simplex}
The set of experiments in Figure~\ref{fig:test1c} correspond to the minimization of the hard instance function for smooth optimization, constrained over the probability simplex. It should be compared to the unconstrained version in Figure~\ref{fig:test1}. As predicted, we observe that the presence of constraints decreases the effect of error accumulation as the boundary of the feasible set limits the variance. Given the low variance due to the constraints, the effect of \rsd~ is less evident for this batch of experiments.

\subsection{Regression on Epileptic Seizure Dataset}

\begin{figure*}[t]
\centering
 \subfigure[$\sigma_{\eta}=0$]{\includegraphics[width=0.24\textwidth]{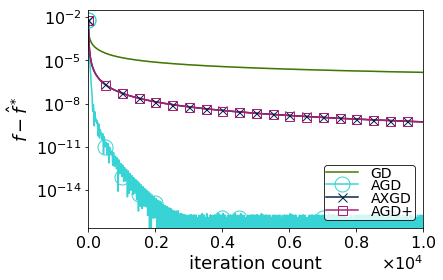}\label{fig:es_lin_l1_0}}
 \subfigure[$\sigma_{\eta}=10^{-5}$]{\includegraphics[width=0.24\textwidth]{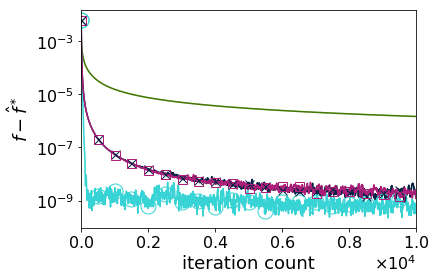}\label{fig:es_lin_l1_1e-5}}
 \subfigure[$\sigma_{\eta}=10^{-3}$]{\includegraphics[width=0.24\textwidth]{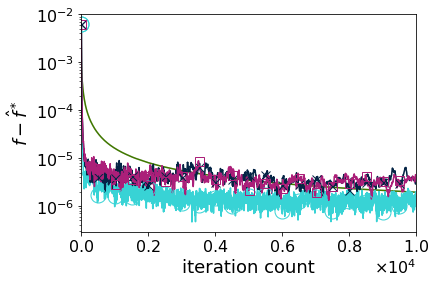}\label{fig:es_lin_l1_1e-3}}
 \subfigure[$\sigma_{\eta}=10^{-1}$]{\includegraphics[width=0.24\textwidth]{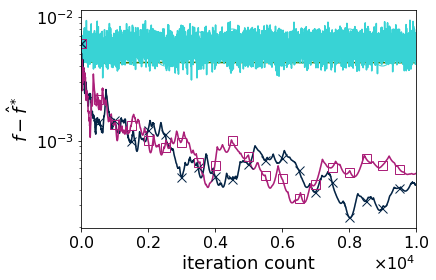}\label{fig:es_lin_l1_1e-1}}
\subfigure[$\sigma_{\eta}=0$]{\includegraphics[width=0.24\textwidth]{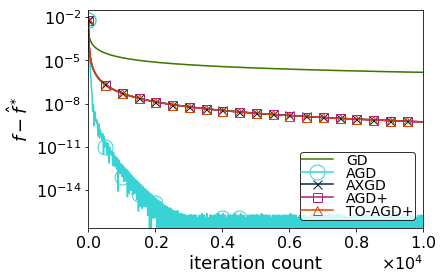}\label{fig:r_es_lin_l1_0_sr}}
 \subfigure[$\sigma_{\eta}=10^{-5}$]{\includegraphics[width=0.24\textwidth]{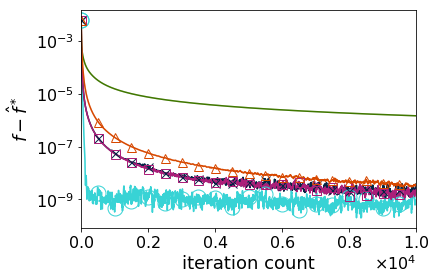}\label{fig:r_es_lin_l1_1e-5_sr}}
 \subfigure[$\sigma_{\eta}=10^{-3}$]{\includegraphics[width=0.24\textwidth]{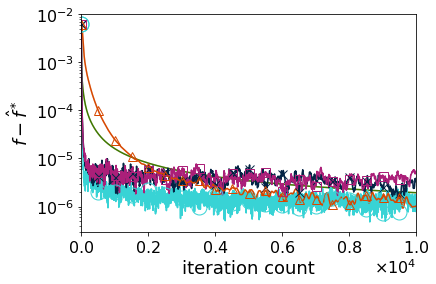}\label{fig:r_es_lin_l1_1e-3_sr}}
 \subfigure[$\sigma_{\eta}=10^{-1}$]{\includegraphics[width=0.24\textwidth]{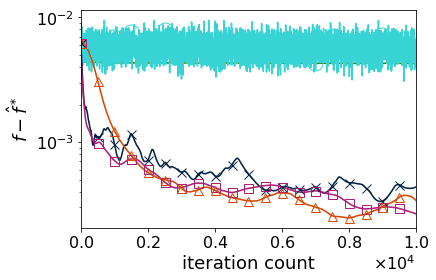}\label{fig:r_es_lin_l1_1e-1_sr}}
\subfigure[$\sigma_{\eta}=0$]{\includegraphics[width=0.24\textwidth]{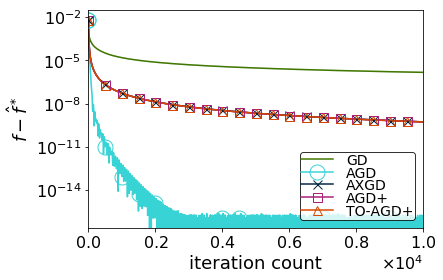}\label{fig:r_es_lin_l1_0}}
 \subfigure[$\sigma_{\eta}=10^{-5}$]{\includegraphics[width=0.24\textwidth]{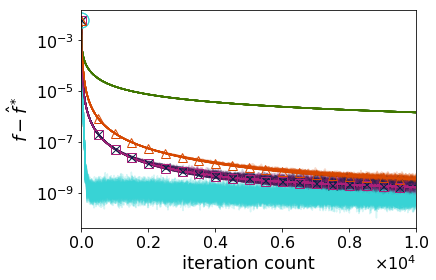}\label{fig:r_es_lin_l1_1e-5}}
 \subfigure[$\sigma_{\eta}=10^{-3}$]{\includegraphics[width=0.24\textwidth]{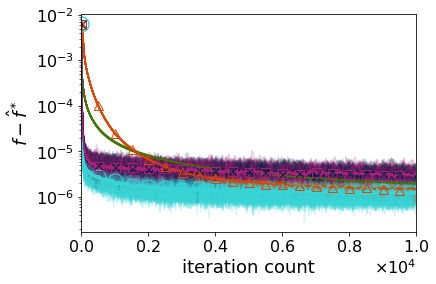}\label{fig:r_es_lin_l1_1e-3}}
 \subfigure[$\sigma_{\eta}=10^{-1}$]{\includegraphics[width=0.24\textwidth]{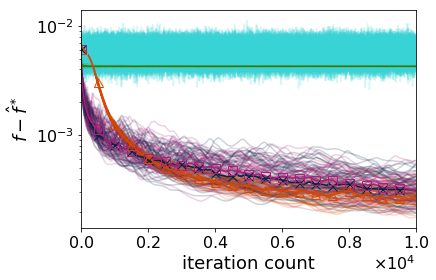}\label{fig:r_es_lin_l1_1e-1}}
\caption{Performance of \gd~and accelerated algorithms (\agd, \axgd, \agdp, \textsc{to-}\agdp)  for linear regression over $\ell_1$-ball (LASSO) on Epileptic Seizure Recognition Dataset~\cite{andrzejak2001indications} \protect\subref{fig:es_lin_l1_0}-\protect\subref{fig:es_lin_l1_1e-1} without noise reduction;  \protect\subref{fig:r_es_lin_l1_0}-\protect\subref{fig:r_es_lin_l1_1e-1} sample run with \textsc{to}\agdp, \rsd~ and \rsd-2; and \protect\subref{fig:r_es_lin_l1_0}-\protect\subref{fig:r_es_lin_l1_1e-1} 50 repeated runs and the median with \textsc{to}\agdp, \rsd~ and \rsd-2.}
\label{fig:test2}
\end{figure*}

For the second set of experiments, we used the Epileptic Seizure Recognition Dataset~\cite{andrzejak2001indications} obtained from~\cite{Lichman:2013}. The dataset consists of brain activity EEG recordings for 100  patients at different time points and in five different states, of which only one indicates epileptic seizure. The dataset contains 11500 rows and 179 columns, of which the first 178 columns are features while the last column indicates whether the patient was in seizure. Before running the experiments, we standardize the data using a standard preprocessing function from the Python Sci-kit library.

\paragraph{Linear regression and LASSO.}
We performed linear regression on the considered dataset mainly to illustrate the performance of the algorithms in the bounded regime (for $\ell_1$-constraints -- LASSO), which we discuss here. The results for standard, unconstrained linear regression are similar to the results for logistic regression discussed below and are  omitted. 

Fig.~\ref{fig:es_lin_l1_0}-\ref{fig:es_lin_l1_1e-1} shows the performance of \gd~and accelerated algorithms for $\ell_1$-constrained linear regression (LASSO) on the Epileptic Seizure Recognition Dataset. Interestingly, the experiments suggest that there are cases when \agd~performes better than \axgd~and \agdp. In particular, in the noiseless (Fig.~\ref{fig:es_lin_l1_0}) and the low-noise (Fig.~\ref{fig:es_lin_l1_1e-5}) settings, \agd~performs much better than worst-case and converges faster than \axgd~and \agdp. 

However, the faster convergence comes at the expense of lower stability to noise as the noise becomes higher. Specifically, as the noise is increased, \agd~performs only marginally better and with higher variance than \axgd~and \agdp~(Fig.~\ref{fig:es_lin_l1_1e-3}), and stabilizes to much higher mean and variance in the very high-noise setting (Fig.~\ref{fig:cyc_uc_1e-1}). 

Intuitively, ``greedy'' gradient steps that \agd~takes may reduce the function value significantly and lead to faster convergence in the noiseless and low-noise settings, while making the convergence very sensitive to the noise from the last iteration, as the gradient steps only depend on the last seen (noisy) gradient. In contrast, \axgd~and \agdp~are more stable to noise, since both of their per-iteration steps depend on the aggregate gradient (and thus, aggregate noise) information.

As expected from the analytical results from Section~\ref{sec:algorithm}, restart and slow-down does not noticeably improve the mean error of the algorithms (Fig.~\ref{fig:r_es_lin_l1_0_sr}-\ref{fig:r_es_lin_l1_1e-1_sr}, \ref{fig:r_es_lin_l1_0}-\ref{fig:r_es_lin_l1_1e-1}). However, in agreement with the analysis, it can reduce the error variance in the high-noise-variance setting (Fig.~\ref{fig:r_es_lin_l1_1e-1}). We also note that \textsc{to-}\agdp~is more stable over the repeated methods' execution, at the expense of slower initial convergence. 

\begin{figure*}[ht!]
\centering
 \subfigure[$\sigma_{\eta}=0$]{\includegraphics[width=0.24\textwidth]{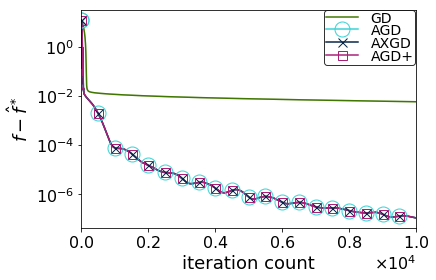}\label{fig:es_log_uc_0}}
 \subfigure[$\sigma_{\eta}=10^{-5}$]{\includegraphics[width=0.24\textwidth]{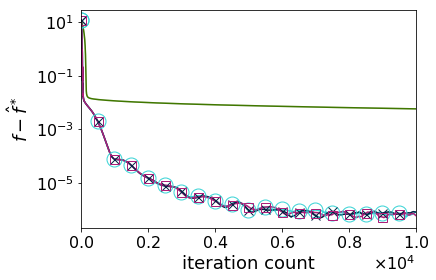}\label{fig:es_log_uc_1e-5}}
 \subfigure[$\sigma_{\eta}=10^{-3}$]{\includegraphics[width=0.24\textwidth]{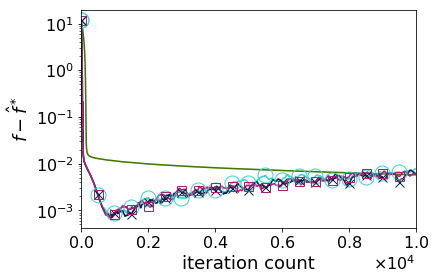}\label{fig:es_log_uc_1e-3}}
 \subfigure[$\sigma_{\eta}=10^{-1}$]{\includegraphics[width=0.24\textwidth]{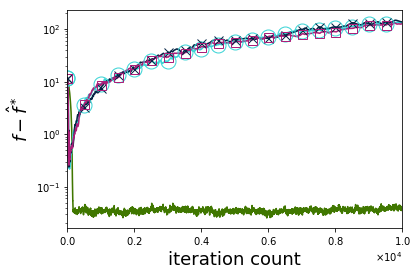}\label{fig:es_log_uc_1e-1}}
\subfigure[$\sigma_{\eta}=0$]{\includegraphics[width=0.24\textwidth]{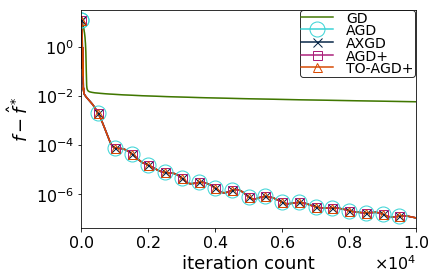}\label{fig:r_es_log_uc_0_sr}}
 \subfigure[$\sigma_{\eta}=10^{-5}$]{\includegraphics[width=0.24\textwidth]{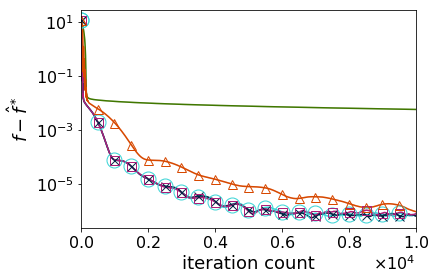}\label{fig:r_es_log_uc_1e-5_sr}}
 \subfigure[$\sigma_{\eta}=10^{-3}$]{\includegraphics[width=0.24\textwidth]{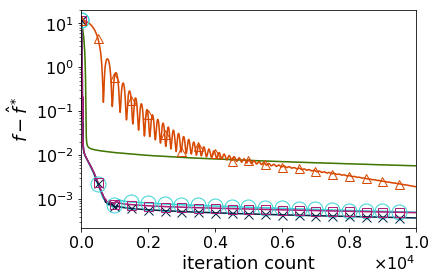}\label{fig:r_es_log_uc_1e-3_sr}}
 \subfigure[$\sigma_{\eta}=10^{-1}$]{\includegraphics[width=0.24\textwidth]{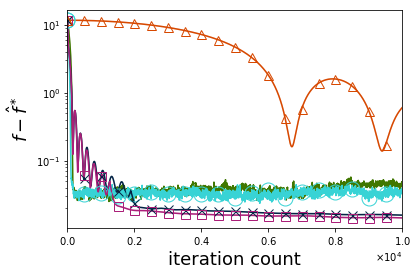}\label{fig:r_es_log_uc_1e-1_sr}}
\subfigure[$\sigma_{\eta}=0$]{\includegraphics[width=0.24\textwidth]{r_es_log_uc_0_sr.png}\label{fig:r_es_log_uc_0}}
 \subfigure[$\sigma_{\eta}=10^{-5}$]{\includegraphics[width=0.24\textwidth]{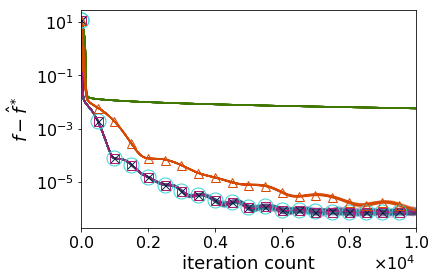}\label{fig:r_es_log_uc_1e-5}}
 \subfigure[$\sigma_{\eta}=10^{-3}$]{\includegraphics[width=0.24\textwidth]{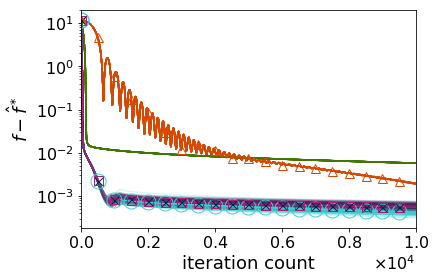}\label{fig:r_es_log_uc_1e-3}}
 \subfigure[$\sigma_{\eta}=10^{-1}$]{\includegraphics[width=0.24\textwidth]{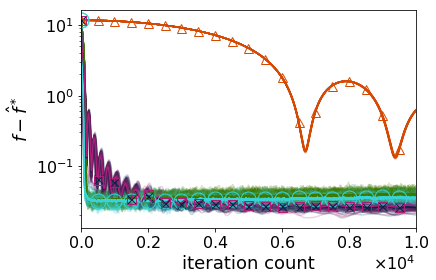}\label{fig:r_es_log_uc_1e-1}}
\caption{Performance of gradient descent (\gd) and accelerated algorithms (\agd, \axgd, \agdp, \textsc{to-}\agdp) for logistic regression on Epileptic Seizure Recognition Dataset~\cite{andrzejak2001indications}  \protect\subref{fig:es_log_uc_0}-\protect\subref{fig:es_log_uc_1e-1} without noise reduction;  \protect\subref{fig:r_es_log_uc_0_sr}-\protect\subref{fig:r_es_log_uc_1e-1_sr} sample run with noise reduction; and \protect\subref{fig:r_es_log_uc_0}-\protect\subref{fig:r_es_log_uc_1e-1} 50 repeated runs and the median run with noise reduction.}
\label{fig:log-reg}
\end{figure*}
\paragraph{Logistic regression.}
Finally, we evaluated the performance of the accelerated algorithms and \gd~for (unregularized) logistic regression on the Epileptic Seizure Recognition Dataset. The results are shown in Fig.~\ref{fig:log-reg}.

Similar as in the case of unconstrained minimization from the beginning of this section, in the noiseless and low-noise settings (Fig.~\ref{fig:es_log_uc_0}, \ref{fig:es_log_uc_1e-5}) all accelerated algorithms perform similarly and restart and slow-down does not lead to any noticeable improvements or degradation (Fig.~\ref{fig:r_es_log_uc_0_sr}, \ref{fig:r_es_log_uc_0}, \ref{fig:r_es_log_uc_1e-5_sr}, \ref{fig:r_es_log_uc_1e-5}). Once the noise is high enough (Fig.~\ref{fig:es_log_uc_1e-3}, \ref{fig:es_log_uc_1e-1}), all accelerated algorithms begin to accumulate noise, while restart and slow-down stabilize their performance to a low error mean and variance. Interestingly, in all the experiments, when \rsd~and \rsd-2~are employed all accelerated algorithms perform at least as good as \gd~in terms of the error mean and variance. 


\section{Conclusion}\label{sec:conclusion}

This paper presents a new accelerated algorithm together with the analysis of its associated error bounds in the cases when the gradient oracle is corrupted by additive noise. Moreover, motivated by the analytical results, we also provide simple semi-heuristics that restart and slow down the accelerated algorithms to reduce their error mean and variance. Our numerical experiments corroborate the analytical results. 

There are several interesting directions for future work that merit further investigation. For example, restart \& slow-down approaches that do not require the explicit knowledge of the noise variance would be interesting for applications in engineered systems where gradients are estimated from noise-corrupted measurements.

\section*{Acknowledgements}
Part of this work was done while the authors were visiting the Simons Institute for the Theory of Computing. It was partially supported by NSF grant \#CCF-1718342 and by the DIMACS/Simons Collaboration on Bridging Continuous and Discrete Optimization through NSF grant \#CCF-1740425. 
JD and LO would like to thank Guanghui Lan, Pavel Dvurechensky and the anonymous reviewers for useful comments.

The algorithm \agdp~for the noiseless case is due to Michael B. Cohen, who termed it a ``proper extension of Nesterov's method'' and shared it with JD during the Fall 2017 semester at the Simons Institute for the Theory of Computing. JD and LO dedicate this paper to the memory of Michael's brilliance and scholarship.
\bibliography{references}
\bibliographystyle{icml2018}

\end{document}

%% file: prologue.tex
\usepackage{amsmath}
\usepackage{amssymb}
\usepackage{amsthm}
\usepackage{bm}
\usepackage{graphicx}
\usepackage{caption}
\usepackage{comment}

\newcommand{\innp}[1]{\left\langle #1 \right\rangle}

\newcommand{\mA}{\mathbf{A}}

\newcommand{\zeros}{\textbf{0}}
\newcommand{\vx}{\mathbf{x}}
\newcommand{\vxh}{\mathbf{\hat{x}}}

\newcommand{\vy}{\mathbf{y}}
\newcommand{\vz}{\mathbf{z}}
\newcommand{\vv}{\mathbf{v}}
\newcommand{\vw}{\mathbf{w}}
\newcommand{\vvh}{\mathbf{\hat{v}}}
\newcommand{\vb}{\mathbf{b}}

\newcommand{\vu}{\mathbf{u}}
\newcommand{\veta}{\bm{\eta}}
\newcommand{\vxi}{\bm{\xi}}

\newcommand{\defeq}{\stackrel{\mathrm{\scriptscriptstyle def}}{=}}

\newcommand{\tnabla}{\widetilde{\nabla}}

\newcommand{\littlesum}{\mathop{\textstyle\sum}}

\graphicspath{{imgs/}}

\makeatletter
\def\mathcolor#1#{\@mathcolor{#1}}
\def\@mathcolor#1#2#3{%
  \protect\leavevmode
  \begingroup
    \color#1{#2}#3%
  \endgroup
}
\makeatother

\newcommand*{\vsepfbox}[1]{%
  \begingroup
    \sbox0{\fbox{#1}}%
    \setlength{\fboxrule}{0pt}%
    \mbox{\kern-\fboxsep\fbox{\unhbox0}\kern-\fboxsep}%
  \endgroup
}

\theoremstyle{plain} \numberwithin{equation}{section}
\newtheorem{theorem}{Theorem}[section]
\numberwithin{theorem}{section}
\newtheorem{corollary}[theorem]{Corollary}

\newtheorem{lemma}[theorem]{Lemma}
\newtheorem{proposition}[theorem]{Proposition}

\newtheorem{fact}[theorem]{Fact}
\theoremstyle{definition}
\newtheorem{definition}[theorem]{Definition}

\newtheorem{remark}[theorem]{Remark}

\newcommand{\axgd}{\textsc{axgd}}
\newcommand{\agdp}{\textsc{agd}$+$}
\newcommand{\magdp}{$\mu$\textsc{agd}$+$}
\newcommand{\agd}{\textsc{agd}}
\newcommand{\gd}{\textsc{gd}}

\newcommand{\acsa}{\textsc{ac-sa}}
\newcommand{\rsd}{\textsc{Restart+SlowDown}}

\DeclareMathOperator*{\argmin}{argmin}